\documentclass{amsart}
\pdfoutput=1

\usepackage{amssymb}
\usepackage{microtype}
\usepackage{tikz}
\usetikzlibrary{arrows}
\usepackage{booktabs}
\usepackage[pdftex,colorlinks,citecolor=black,linkcolor=black,urlcolor=black,bookmarks=false]{hyperref}
\def\MR#1{\href{http://www.ams.org/mathscinet-getitem?mr=#1}{MR#1}}

\newtheorem{theorem}{Theorem}[section]
\newtheorem{proposition}[theorem]{Proposition}

\newtheorem{corollary}[theorem]{Corollary}

\newtheorem{fact}[theorem]{Fact}

\numberwithin{figure}{section}
\numberwithin{equation}{section}
\numberwithin{table}{section}

\newcommand{\R}{\mathbb{R}}
\newcommand{\C}{\mathbb{C}}

\newcommand{\E}{\mathbb{E}}
\newcommand{\prob}{\mathbb{P}}

\newcommand{\vol}{\operatorname{vol}}
\newcommand{\var}{\operatorname{var}}

\newcommand{\A}{\mathfrak{a}}
\newcommand{\B}{\mathfrak{b}}

\DeclareMathOperator{\tr}{tr}

\title{A central limit theorem for integrals of random waves}

\author{Matthew de Courcy-Ireland}
\address{Department of Mathematics\\
Princeton University\\
Princeton NJ 08544
and Institute of Mathematics, EPFL, Lausanne CH-1015} \email{mdc4@math.princeton.edu}

\author{Marius Lemm}

\address{School of Mathematics\\
Institute for Advanced Study\\
Princeton NJ 08540\\
and Department of Mathematics\\
Harvard University\\
Cambridge MA 02138
} \email{mlemm@math.harvard.edu}

\date{March 14, 2019}

\begin{document}

\begin{abstract}
We derive a central limit theorem for the mean-square of random waves in the high-frequency limit over shrinking sets. Our proof applies to any compact Riemannian manifold of arbitrary dimension, thanks to the universality of the local Weyl law. The key technical step is an estimate capturing some cancellation in a triple integral of Bessel functions, which we achieve using Gegenbauer's addition formula.
\end{abstract}

\maketitle

\section{Introduction} \label{sec:intro}

The goal of this paper is to prove a central limit theorem for $\int_B \phi^2$, where $\phi: M \rightarrow \R$ is a random wave and the ball $B$ may shrink with the wavelength of $\phi$. 
On any compact manifold $M$ with a Riemannian metric and corresponding Laplace operator $\Delta$, the random functions we have in mind are given by
\begin{equation} \label{eqn:def-sum}
\phi(x) = \sum_j c_j \phi_j(x)
\end{equation}
where the eigenfunctions $\phi_j$ solve $(\Delta + t_j^2)\phi_j = 0$ with eigenvalues $t_j^2$ in a window  $T - \eta(T) \leq t_j \leq T$.
The coefficients $c_j$ are independent Gaussians of mean 0 and identical variance.  
The choice of variance will disappear when we pass to the standardized random variable
\begin{equation}
Z = Z_B = \frac{\int_B \phi^2 - \E[\int_B \phi^2] }{\sqrt{\var[\int_B \phi^2 ]}},
\end{equation}
which is invariant under scaling $\phi$ by a constant multiple or, what is the same, scaling the variance of the coefficients. Nevertheless, it is natural to choose  the variance of each coefficient $c_j$ inversely proportional to the number of terms in the sum (\ref{eqn:def-sum}), so that $\E[ \int_M \phi^2 ] = 1$.
A standard Gaussian $G$ has characteristic function
\[
\E[ e^{it G} ] = e^{-t^2/2}
\]
and our main result is an estimate comparing $\E[e^{itZ}]$ to this.

\begin{theorem} \label{thm:main}
For any compact Riemannian manifold $M$ of dimension $n \geq 2$, any window $\eta(T) \rightarrow \infty$ with $\eta(T)\max(T^{-1/2},r^{2}T)$
bounded above, and a geodesic ball $B = B_r(z)$ shrinking at such a rate that $rT \rightarrow \infty$ and $(rT)^{n-1}/\eta \rightarrow 0$, we have convergence of characteristic functions
\begin{equation} \label{eqn:clt}
\E\left[ e^{itZ} \right] = e^{-t^2/2} \left( 1 + O\big((rT)^{ -1/2} \big) \right)
\end{equation}
for any fixed parameter $t \in \C$.
\end{theorem}

The characteristic function $\E\big[ e^{itZ} \big]$ is defined for all real $t$ and, once $rT$ is sufficiently large, for any fixed complex $t$.
In particular, Theorem~\ref{thm:main} implies convergence in distribution.
\begin{corollary} \label{cor:clt} (Central Limit Theorem)
In the limit that $T \rightarrow \infty$ with $r \rightarrow 0$, $\eta \rightarrow \infty$ as above, the standardized mean squares $Z_B$ converge in distribution to a Gaussian of mean 0 and variance 1:
\[
Z_B \rightarrow N(0,1).
\]
\end{corollary}
For example, one could take $\eta = \log{T}$ and $r = \log(T)^{\varepsilon}/T$ for a sufficiently small $\varepsilon > 0$.
It is important for our proof that $r \rightarrow 0$, although one might expect the Central Limit Theorem to apply also for sufficiently small fixed radii 
in the limit $T \rightarrow \infty$.
The scale $1/T$ is significant because it is the natural wavelength for Laplace eigenfunctions of frequency $T$, and hence for the random waves (\ref{eqn:def-sum}). We assume that $rT \rightarrow \infty$ so that a length scale $r$ is enough to contain many oscillations of $\phi$, and we have in mind that $rT \rightarrow \infty$ arbitrarily slowly so that one is almost at the wave scale. We do not expect the central limit theorem to hold if $rT$ remains bounded, as we explain in Section~\ref{sec:fail}, so that the wave scale provides a natural barrier.

We now give an overview of our strategy for proving Theorem~\ref{thm:main} and some of the notation we will employ.
Throughout the paper, we write $\sum_j$ as an abbreviation for sums over a given window of frequencies $T-\eta(T) \leq t_j \leq T$. For example, we have already done so in (\ref{eqn:def-sum}).
We use the notations
\[
f(X) = O\big( g(X) \big), \quad f(X) \lesssim g(X)
\]
with the same meaning: there is a constant $C > 0$ such that $|f(X)| \leq C g(X)$ for all sufficiently large values of the variable(s) $X$.
We write $f(X) \asymp g(X)$ if both of the approximate inequalities $f(X) \lesssim g(X)$ and $g(X) \lesssim f(X)$ hold, that is, there are constants $C_1, C_2 > 0$ such that
\[
C_1 g(X) \leq f(X) \leq C_2 g(X).
\]
If the constants $C, C_1, C_2, \ldots$ depend on parameters, we indicate this with subscripts, as in
\[
\log{X} \lesssim_{\varepsilon} X^{\varepsilon}.
\]
The function $f$ is not necessarily real when we write $f(X) \lesssim g(X)$, but we only write $f(X) \asymp g(X)$ when both functions are positive.
In Section~\ref{sec:quad}, we compute the characteristic function of a quadratic form in Gaussian random variables, which includes $\int_B \phi^2$ as a special case.
We give a sufficient condition for such a quadratic form to obey the Central Limit Theorem.
\begin{proposition} \label{prop:quads}
Given a sequence of symmetric, positive definite $N \times N$ matrices $A = A_N$ and standard Gaussian vectors $\mathfrak{z} = (\mathfrak{z}_1, \ldots, \mathfrak{z}_N)$, let $Z$ be the standardized quadratic form
\[
Z = \frac{ \mathfrak{z}^T A \mathfrak{z} - \E\left[ \mathfrak{z}^T A \mathfrak{z} \right] }{\sqrt{\var\left[ \mathfrak{z}^T A \mathfrak{z} \right] } }.
\]
If
\[
\frac{\tr(A^3)}{\tr(A^2)^{3/2} } \rightarrow 0
\]
as $N \rightarrow \infty$,
then the characteristic function of $Z$ obeys
\[
\E\left[ e^{itZ} \right] = e^{-t^2/2} \left( 1 + O\left( |t|^3 \frac{ \tr(A^3) }{\tr(A^2)^{3/2}} \right) \right).
\]
\end{proposition}
This can be seen as an instance of Lyapunov's criterion for deriving a central limit theorem \cite[p.362]{Bil}: If $S = X_1 + \ldots X_N$ is a sum of independent random variables, perhaps not with the same distribution, and $\E[|S|^{2+\delta}]$ is small compared to $\E[|S|^2]^{1+\delta/2}$ for some $\delta > 0$, then $S$ converges to a Gaussian as $N \rightarrow \infty$. Proposition~\ref{prop:quads} takes $\delta = 1$, and one diagonalizes the quadratic form to obtain a sum of independent random variables. For a quadratic form in Gaussians, one can replace the general argument of Lyapunov's criterion with an explicit computation of the characteristic function.

In the case of $\int_B \phi^2$, the matrix $A$ has a special form that allows the traces $\tr(A^p)$ to be computed in terms of the \emph{two-point function}
\begin{equation} \label{eqn:kxy}
K(x,y) = \sum_j \phi_j(x)\phi_j(y).
\end{equation}
We have $\E[\phi(x)\phi(y)] = \var[c]K(x,y)$ so that, after choosing the variance of the coefficients, $K(x,y)$ gives the correlations between values of $\phi$ at different points.
In Section~\ref{sec:moments-K}, we reformulate the criterion of Proposition~\ref{prop:quads} as
\begin{proposition} \label{prop:kmom}
If $\phi = \sum_j c_j \phi_j$ is a Gaussian random function with two-point function $K(x,x')$
satisfying
\[
\frac{ \int_B \int_B \int_B K(x_1,x_2)K(x_2,x_3)K(x_3,x_1) dx_3 dx_2 dx_1 }{\left( \int_B \int_B K(x,x')^2 dx dx' \right)^{3/2}} \rightarrow 0
\]
as the ball $B$ shrinks, then the local integral $\int_B \phi^2$ obeys the Central Limit Theorem in the form
\[
\E\left[ e^{itZ} \right] = e^{-t^2/2} \left( 1 + O\left( |t|^3 \frac{ \int_B \int_B \int_B K(x_1,x_2)K(x_2,x_3)K(x_3,x_1) dx_3 dx_2 dx_1 }{\left( \int_B \int_B K(x,x')^2 dx dx' \right)^{3/2}}  \right) \right).
\]
\end{proposition}
Theorem~\ref{thm:main} is then deduced from the following two estimates:
\begin{theorem} \label{thm:2nd-mom}
In any dimension $n \geq 2$,
\[
\int_B \int_B K(x,x')^2 dxdx' \asymp (rT)^{-n+1} \big( \vol(B) T^{n-1} \eta \big)^2.
\]
\end{theorem}
\begin{theorem} \label{thm:3rd-mom}
In dimensions $n \geq 3$,
\[
\int_B \int_B \int_B K(x_1,x_2)K(x_2,x_3)K(x_3,x_1) dx_3 dx_2 dx_1 \lesssim (rT)^{-3n/2+1} \big( \vol(B) T^{n-1} \eta \big)^3.
\]
\end{theorem}
In Section~\ref{sec:weyl}, we use the local Weyl law to approximate $K(x,y)$ and it is for this purpose that Theorem~\ref{thm:main} assumes $\eta \lesssim T^{1/2}$. 
In Section~\ref{sec:variance}, we prove Theorem~\ref{thm:2nd-mom}, and in particular give a lower bound on the denominator $\int_B \int_B K^2$, which is essentially the variance of $\int_B \phi^2$.
In Section~\ref{sec:euclid}, the hypothesis $\eta \lesssim r^{-2}T^{-1}$ comes into play as we approximate the triple integral in the numerator by a Euclidean version. For $n \geq 3$, we bound this integral from above in Sections~\ref{sec:gegenbauer} and \ref{sec:bessel}, completing the proof of Theorem~\ref{thm:3rd-mom} in Section~\ref{sec:complete} for dimensions $n \geq 3$. 
We prove the two-dimensional case of Theorem~\ref{thm:main} in Section~\ref{sec:s2} by a different argument: reducing to the case where $M$ is the round sphere.

For context, we recall Berry's Random Wave Model \cite{B}, which uses monochromatic random waves of the form (\ref{eqn:def-sum}) as a stand-in to make predictions about non-random eigenfunctions. This is expected to be a good approximation for chaotic systems. The model applies to high-frequency Laplace eigenfunctions on a manifold of negative curvature and suggests that, at an appropriate scale, they should be uniformly distributed. Related to this is the Quantum Unique Ergodicity conjecture of Rudnick and Sarnak \cite{RS} that on a negatively curved manifold $M$,
\begin{equation*}
\int_A \phi_{\lambda}^2 d\vol \rightarrow \vol(A)
\end{equation*}
for any fixed measurable subset $A$ of $M$ and any sequence of Laplace eigenfunctions $\phi_{\lambda}$ with growing eigenvalue $\lambda \rightarrow \infty$. 
This is one of our motivations for considering local integrals $\int_A \phi^2$, and the limit $\vol(A)$ corresponds to the expected value of $\int_A \phi^2$ when $\phi$ is taken at random.
The eigenvalue equation $(\Delta + \lambda)\phi = 0$ imparts a quantum interpretation, $|\phi|^2$ being the probability density of a single quantum particle in $M$. 

The quantum ergodicity theorem proved by Shnirelman \cite{Sh1,Sh2}, Colin de Verdi\`{e}re \cite{CdV}, and Zelditch \cite{Z} shows that negative curvature implies convergence along a full subsequence of eigenfunctions, or equivalently on average over the eigenfunctions, but allows many other subsequential limits besides the uniform measure.
Although it remains unknown whether the uniform measure is the only possibility,  work of Anantharaman \cite{A}, Anantharaman-Nonnenmacher \cite{AN}, Anantharaman-Silberman \cite{AS}, and Dyatlov-Jin \cite{DJ} places significant constraints on the measures that arise as quantum limits in general.
In examples of arithmetic origin, QUE has been proved in work of Lindenstrauss \cite{L1,L2}, and Bourgain-Lindenstrauss \cite{BouLi}, Jakobson \cite{J}, Holowinsky \cite{H}, Holowinsky-Soundararajan \cite{HS}.

The deterministic problem being so difficult, it is of interest to randomize the function $\phi$ and study the fluctuations in $\int_A \phi^2$. Our main result shows that for $A$ shrinking at a certain rate, one can expect central limit statistics. 

The central limit theorem controls deviations of size $``\sqrt{\var}"$ in the sense that
\[
\prob(|X - \E[X]| \geq y\sqrt{\var[X]}) \lesssim e^{-y^2/2}
\]
for any fixed $y$.
We also obtain the following concentration estimate at other scales.
\begin{theorem} \label{thm:tail}
If $y(rT)^{-1/6}$ is sufficiently small in terms of the implicit constants in the previous theorems, then
\[
\prob(|X - \E[X]| \geq y\sqrt{\var[X]}) \lesssim e^{-c y^2}
\]
where $c > 0$ is a numerical constant and $X = \int_B \phi^2$ as above.
In particular, the conclusion holds if $y(rT)^{-1/6} \rightarrow 0$ as $rT \rightarrow \infty$.
\end{theorem}

Han and Tacy proved deviation bounds in a related setting in equations (4.4), (4.5) of \cite{HT}, which was a motivation for the present work.

\section{Characteristic function of a quadratic form in Gaussians} \label{sec:quad}

In this section, we prove Proposition~\ref{prop:quads}.
Our quantity of interest $Z$ is a quadratic form in the random Gaussian coefficients $c_j$, and therefore has an explicit characteristic function. Indeed, expanding the square in $\phi^2$ shows that
\begin{equation}
 \int_B \phi^2 = \sum_j \sum_k c_j c_k  \int_B \phi_j \phi_k
\end{equation}
Write $c_j = \sqrt{\var[c]} \mathfrak{z}_j$, where $\mathfrak{z}_j$ is a centered Gaussian of unit variance and $\var[c]$ denotes the common variance $\var[c_j]$ among all $j$. 
Writing $\mathfrak{z}$ for the vector with entries $\mathfrak{z}_j$, we have 
\[
\int_B \phi^2 = \mathfrak{z}^T A \mathfrak{z}
\]
where the matrix $A$ has entries
\begin{equation}
A_{jk} = \var[c]  \int_B \phi_j \phi_k.
\end{equation}
The moment generating function can be computed explicitly by diagonalizing this symmetric matrix $A$. Write $A = U^T DU$, so that
\begin{equation}
 \mathfrak{z}^T A\mathfrak{z} = (U\mathfrak{z})^T D (U\mathfrak{z}) = \sum_j \lambda_j y_j^2
\end{equation}
where $\lambda_j$ are the eigenvalues of $A$.
For an orthogonal matrix $U$ and a standard Gaussian vector $\mathfrak{z}$, the vector $y = U\mathfrak{z}$ also has a standard Gaussian distribution. For $s \geq 0$ small enough that $1 - 2s\lambda_j > 0$ for all $j$, independence of the random variables $y_j$ implies that
\begin{align*}
\E\left[e^{s \mathfrak{z}^T A\mathfrak{z}}\right] &= \E\left[ e^{s \sum_j \lambda_j y_j^2} \right] = \prod_j \E\left[ e^{s \lambda_j y_j^2} \right]
\end{align*}
which is a product of Gaussian integrals. These can be evaluated as
\[
\E\left[ e^{s\lambda y^2} \right] = \int_{-\infty}^{\infty} e^{s\lambda y^2} e^{-y^2/2} \frac{dy}{\sqrt{2\pi}} = (1-2s\lambda)^{-1/2}.
\]

We summarize this as
\begin{proposition}
If $A$ is a symmetric matrix with eigenvalues $\lambda_j \geq 0$, then the moment generating function of the quadratic form $\mathfrak{z}^T A \mathfrak{z}$ in Gaussian random variables $\mathfrak{z}_j$ is
\begin{equation} \label{eqn:mgf}
\E\left[e^{s \mathfrak{z}^T A\mathfrak{z}}\right] = \prod_{j} \left(1 - 2s\lambda_j \right)^{-1/2}
\end{equation}
where both sides are defined for $s$ sufficiently small that $1-2s\lambda_{\max} > 0$, $\lambda_{\max}$ being the largest eigenvalue of $A$.
\end{proposition}

Note that, by differentiating the moment generating function,
\begin{align*}
\E[\mathfrak{z}^T A \mathfrak{z} ] &= \tr(A) =  \sum_j \lambda_j \\
\var[\mathfrak{z}^T A\mathfrak{z} ] &= 2\tr(A^2) = 2\sum_j \lambda_j^2
\end{align*}
Indeed, let $X = \mathfrak{z}^T A\mathfrak{z}$ and differentiate with respect to $s$ starting from
\[
g(s) = \log \E \left[ e^{sX} \right] = \frac{1}{2} \sum_j -\log(1-2s\lambda_j)
\]
to
\[
\E[ X e^{sX} ] = \E[ e^{sX} ] \frac{1}{2} \sum_j \frac{2\lambda_j}{1- 2s\lambda_j}
\]
and then to
\[
\E[X^2 e^{sX} ] = \E[ X e^{sX} ] \frac{1}{2} \sum_j \frac{2\lambda_j}{1- 2s\lambda_j} + \E[e^{sX} ] \frac{1}{2} \sum_j \frac{ (2\lambda_j)^2}{(1-2s\lambda_j)^2}.
\]
Taking $s=0$ gives $\E[X] = \sum \lambda_j$ and then $\E[X^2] = (\E[X])^2 + 2\sum \lambda_j^2$. Thus, as claimed,
\begin{equation}
 \var[X] = \E[X^2] - \E[X]^2 = 2 \sum_j \lambda_j^2.
\end{equation}
Passing to the standardized random variable $Z$, write
\[
\sigma^2 = \var\left[\int_B \phi^2 \right]
\]
so that
\[
\E[e^{sZ}] = \E\left[\exp\left( (s/\sigma) \mathfrak{z}^T A\mathfrak{z} \right) \right] \exp\left(-(s/\sigma) \E[\mathfrak{z}^T A\mathfrak{z} ] \right).
\]
Taking logarithms,
\[
\log \E\left[ e^{sZ} \right] = \sum_j -\frac{1}{2} \log\left( 1 - 2 \frac{s}{\sigma} \lambda_j \right) - \frac{s}{\sigma} \sum \lambda_j .
\]
Expanding the logarithm $-\log(1-x) = \sum x^p/p$, the linear term cancels:
\begin{align*}
\log \E\left[ e^{sZ} \right] &= \frac{1}{2} \frac{4s^2}{2\sigma^2} \sum_j \lambda_j^2 + \frac{1}{2} \sum_{p=3}^{\infty} \left( \frac{\sum \lambda_j^p}{\sigma^p} \right) \frac{(2s)^p}{p} \\
&= \frac{s^2}{2} + \sum_{p=3}^{\infty} \left( \frac{\sum \lambda_j^p}{\sigma^p} \right) \frac{(2s)^p}{2p} .
\end{align*}
Now we take $s=it$ and standardize. Shifting and scaling and starting from
\begin{equation*}
\E\left[ e^{it\mathfrak{z}^T A\mathfrak{z}} \right] = \prod_{j=1}^{N} (1 - 2it\lambda_j)^{-1/2} = \exp\left(\sum_{j=1}^{N}-\frac{1}{2}\log(1-2it\lambda_j) \right)
\end{equation*}
we find that the characteristic function of the standardized quantity
\[
Z = \frac{ \mathfrak{z}^T A \mathfrak{z} - \E\left[ \mathfrak{z}^T A\mathfrak{z} \right] }{\var[\mathfrak{z}^T A\mathfrak{z} ]^{1/2} }
\]
is
\begin{equation} \label{eqn:charf}
\E[ e^{itZ} ] = \exp\left(\sum_{p=2}^{\infty} \tr(A^p) \frac{2^{p-1}i^p t^p}{p \sigma^p}\right) = e^{-t^2 /2}\exp\left(\sum_{p=3}^{\infty} \frac{ \tr(A^p) }{\sigma^p} \frac{2^{p-1}i^p t^p}{p} \right).
\end{equation}
The factor $e^{-t^2/2}$ is the characteristic function of a standard Gaussian. To show that $Z$ converges to a Gaussian, we will estimate the traces $\tr(A^p)$ and show that they are negligible compared to $\sigma^p = (2\tr(A^2))^{p/2}$.

We complete the proof of Proposition~\ref{prop:quads} by using the third moment to bound the higher ones. 
Since $p$-norms are monotone, we have an upper bound
\[
\tr(A^p) = \sum \lambda_j^p \leq \left( \sum \lambda_j^3 \right)^{p/3} = \tr(A^3)^{p/3}.
\]
Therefore, summing a geometric series,
\[
\left| \sum_{p=3}^{\infty} \tr(A^p) \frac{2^{p-1} i^p t^p}{p\sigma^p} \right| \leq \sum_{p=3}^{\infty} \left( \frac{ 2|t|\tr\big( A^3 \big)^{1/3} }{\big(2 \tr(A^2) \big)^{1/2}}\right)^p \lesssim |t|^3 \frac{ \tr(A^3) }{\tr(A^2)^{3/2}}
\]
where the series converges for any fixed $t$ as long as
\[
\frac{\tr(A^3)^{1/3}}{\tr(A^2)^{1/2} } \rightarrow 0.
\]
Here, as in the rest of the paper, we write $\lesssim$ to denote inequality up to a constant, and $\rightarrow$ refers to limits where $rT \rightarrow \infty$ as described in the Introduction.
Thus an upper bound on the third moment and a lower bound on the second are enough to control all others.
Combining this with (\ref{eqn:charf}), we obtain the following estimate for the characteristic function, as stated in the Introduction.
\newtheorem*{prop:quads}{Proposition \ref{prop:quads}}
\begin{prop:quads}
Given a sequence of symmetric, positive definite $N \times N$ matrices $A = A_N$, let $Z$ be the standardized quadratic form
\[
Z = \frac{ \mathfrak{z}^T A \mathfrak{z} - \E\left[ \mathfrak{z}^T A \mathfrak{z} \right] }{\sqrt{\var\left[ \mathfrak{z}^T A \mathfrak{z} \right] } }.
\]
If
\[
\frac{\tr(A^3)}{\tr(A^2)^{3/2} } \rightarrow 0 \quad \text{as} \ N \rightarrow \infty
\]
then the characteristic function obeys
\[
\E\left[ e^{itZ} \right] = e^{-t^2/2} \left( 1 + O\left( |t|^3 \frac{ \tr(A^3) }{\tr(A^2)^{3/2}} \right) \right).
\]
\end{prop:quads}
The rest of our work is to bound the second and third moments.

\section{Moments and the two-point function} \label{sec:moments-K}

We now show how Proposition~\ref{prop:kmom} is a special case of Proposition~\ref{prop:quads}.
For any quadratic form $\mathfrak{z}^T A\mathfrak{z}$ and any $p \geq 2$, we have
\[
\sum_j \lambda_j^p = \tr(A^p) = \sum_k A^{(p)}_{kk} = \sum_k \sum_{j_1} \cdots \sum_{j_{p-1}} A_{kj_1} A_{j_1 j_2} \cdots A_{j_{p-1} k}
\]
but for the particular one $\int_B \phi^2$, we can express the traces in terms of the kernel
\[
K(x,x') = \sum_j \phi_j(x)\phi_j(x').
\]
The matrix entries are $A_{jk} = \var[c]  \int_B \phi_j \phi_k$. By writing the product of $p$ integrals over $B$ as a single integral over $B^p$, we obtain
\begin{align*}
& \int_B \phi_k(x_1) \phi_{j_1}(x_1) dx_1 \int_B \phi_{j_1}(x_2)\phi_{j_2}(x_2) dx_2 \cdots \int_B \phi_{j_{p-1}}(x_p) \phi_k(x_p) dx_p \\
&= \int_{B^p} \phi_k(x_1)\phi_{j_1}(x_1) \phi_{j_1}(x_2)\phi_{j_2}(x_2) \cdots \phi_{j_{p-1}}(x_p) \phi_k(x_p) dx_1 \ldots dx_p
\end{align*}
When we take the (finite) sum $\sum_k \sum_{j_1} \cdots \sum_{j_{p-1}}$ under the integral, it factors into $p$ copies of the two-point function:
\begin{align*}
&= \sum_k \sum_{j_1} \cdots \sum_{j_{p-1}} \int_{B^p} \phi_k(x_1)\phi_{j_1}(x_1) \phi_{j_1}(x_2)\phi_{j_2}(x_2) \cdots \phi_{j_{p-1}}(x_p) \phi_k(x_p) dx_1 \ldots dx_p \\
&= \int_{B^p} \sum_k \phi_k(x_1)\phi_k(x_p) \sum_{j_1} \phi_{j_1}(x_1)\phi_{j_1}(x_2) \cdots \sum_{j_{p-1}} \phi_{j_{p-1}}(x_{p-1}) \phi_{j_{p-1}}(x_{p}) dx_1 \ldots dx_p \\
&= \int_{B^p} K(x_1,x_p) K(x_1,x_2) \cdots K(x_{p-1},x_p)  dx_1 \ldots dx_p \\
&= \int_{B^p} \prod_{i \ \text{mod} \ p } K(x_i, x_{i+1}) \ dx_1 \ldots dx_p
\end{align*}
With the extra factor of $\var[c]$ from the entries $A_{jk}$, the trace becomes
\begin{equation} \label{eq:p-moment}
\sum_j \lambda_j^p = \var[c]^p \int_B \prod_{i \bmod p} K(x_i, x_{i+1}) \ dx_1 \ldots dx_p
\end{equation}

In particular,
\[
\frac{\tr(A^3)}{\tr(A^2)^{3/2}} = \frac{ \int_B \int_B \int_B K(x_1,x_2)K(x_2,x_3)K(x_3,x_1) dx_3 dx_2 dx_1 }{\left( \int_B \int_B K(x,x')^2 dx dx' \right)^{3/2}}
\]
We summarize this discussion in the following restatement of Proposition~\ref{prop:kmom}, which may be helpful in other examples besides the monochromatic ensemble.
\begin{proposition}
If $\phi = \sum_j c_j \phi_j$ is a Gaussian random function with two-point function
satisfying
\[
\frac{ \int_B \int_B \int_B K(x_1,x_2)K(x_2,x_3)K(x_3,x_1) dx_3 dx_2 dx_1 }{\left( \int_B \int_B K(x,x')^2 dx dx' \right)^{3/2}} \rightarrow 0
\]
as the ball $B$ shrinks, then the local integral $\int_B \phi^2$ obeys the Central Limit Theorem.
\end{proposition}
In our example, we will use semiclassics to estimate $K(x,x')$ and verify the assumption that this ratio of integrals tends to 0 as $rT \rightarrow \infty$. The lower bound on the denominator is provided by
\newtheorem*{thm:2nd-mom}{Theorem \ref{thm:2nd-mom}}
\begin{thm:2nd-mom}
\[
\int_B \int_B K(x,x')^2 dxdx' \asymp (rT)^{-n+1} \big( \vol(B) T^{n-1} \eta \big)^2
\]
\end{thm:2nd-mom}
while the upper bound on the numerator is
\newtheorem*{thm:3rd-mom}{Theorem \ref{thm:3rd-mom}}
\begin{thm:3rd-mom}
\[
\int_B \int_B \int_B K(x_1,x_2)K(x_2,x_3)K(x_3,x_1) dx_3 dx_2 dx_1 \lesssim (rT)^{-3n/2+1} \big( \vol(B) T^{n-1} \eta \big)^3.
\]
\end{thm:3rd-mom}

\section{Weyl's law} \label{sec:weyl}

The local Weyl law is an estimate for $K(x,x')$ showing that it resembles a Bessel function at the scale $1/T$.
Such results were obtained by H\"{o}rmander in the regime $d(x,y) \lesssim T^{-1}$, but we are interested in distances $d(x,y)$ shrinking arbitrarily slowly as $T \rightarrow \infty$. In this situation, we have the following by work of Canzani and Hanin.
\begin{theorem} \label{thm:canhan} (Canzani-Hanin)
Let $M$ be a compact manifold of dimension $n \geq 2$, with smooth Riemannian metric $g$ and corresponding Laplace eigenfunctions obeying $(\Delta + t_j^2) \phi_j = 0$. Fix any point $x_0 \in M$. For a large parameter $T$, let $B = B_r(x_0)$ where $r = r(T) \rightarrow 0$ arbitrarily slowly as $T \rightarrow \infty$. Then the spectral function
\[
E_T(x,y) = \sum_{t_j \leq T} \phi_j(x) \phi_j(y)
\]
can be written
\[
E_T(x,y) = \left( \frac{T}{2\pi} \right)^n \int_{|\xi|_{g_y}<1} e^{i T \langle \exp_y^{-1}(x), \xi \rangle_{g_y} } \frac{d\xi}{\sqrt{|g_y|}} + R_T(x,y)
\]
where the remainder satisfies
\[
\sup_{x,y \in B} |R_T(x,y)| \lesssim T^{n-1}.
\]
\end{theorem}
H\"{o}rmander had already proved this with a remainder at most $O(T^{n-1})$ for $d(x,y)$ of order $1/T$.
In Theorem 2 from \cite{CH}, Canzani and Hanin improved this in two ways: allowing $d(x,y)$ to shrink more slowly than $1/T$ and, assuming $x_0$ is not self-focal, achieving an error term $o(T^{n-1})$ instead of $O(T^{n-1})$.
Without the assumption on $x_0$, their proof still gives $O(T^{n-1})$ for $d(x,y)$ shrinking arbitrarily slowly, as we proceed to sketch.
Proposition 18 in \cite{CH} is the main technical estimate, and it allows $x_0$ to be self-focal. 
The role of the non-self-focal assumption is to guarantee the estimate (43) in \cite{CH}, which says roughly that there is a constant $c>0$ such that when $x=y$, the remainder is at most $c\varepsilon T^{n-1} + O_{\varepsilon}(T^{n-2})$ for any $\varepsilon > 0$. Without the (non)focal assumption, one still knows by H\"ormander's work that the remainder is $O(T^{n-1})$ on the diagonal and the same arguments from \cite{CH} (in particular Propositions 10 and 11) imply that $R(x,y)$ remains $O(T^{n-1})$ for distances $d(x,y)$ shrinking arbitrarily slowly.

The phase $\langle \exp_y^{-1}(x), \xi \rangle_{g_y}$ is an important aspect of \cite{CH}. It is defined as long as $d(x,y)$ is less than the injectivity radius of $M$, that is, even for distances of order 1. 
In contrast, H\"{o}rmander's ``adapted" phase functions $\psi(x,y,\xi)$ are obtained from the eikonal equation, a differential equation which is not guaranteed to have solutions when $Td(x,y)$ is unbounded.

For a window of frequencies $T-\eta < t_j \leq T$, the kernel is
\[
K(x,y) = E_T(x,y) - E_{T-\eta}(x,y).
\]
The remainders $R_{T-\eta} \lesssim (T-\eta)^{n-1}$ and $R_T \lesssim T^{n-1}$ are both $O(T^{n-1})$.
Therefore Theorem~\ref{thm:canhan} (Theorem 2 from \cite{CH}) implies an estimate for $K(x,y)$. Canzani-Hanin state the result as Theorem 3 in \cite{CH}, but again they are interested in obtaining an improved remainder $o(T^{n-1})$ in the most crisp regime $\eta = 1$ under the assumption that $x,y$ are close to a non-self-focal point $x_0$. In the present article, we avoid making this assumption at the price of taking a slowly growing window $\eta \rightarrow \infty$. 
Rescaling the integral for $E_{T-\eta}$ by $\xi \mapsto T\xi/(T-\eta)$, we have
\[
 \int_{|\xi|<1} e^{i (T-\eta)\langle \exp_y^{-1}(x),\xi \rangle} \frac{d\xi}{\sqrt{|g_y|}} = \left(1-\frac{\eta}{T}\right)^{-n} \int_{|\xi|< 1 - \frac{\eta}{T}} e^{i T \langle \exp_y^{-1}(x),\xi \rangle} \frac{d\xi}{\sqrt{|g_y|}}
\]
where we have suppressed the subscripts with the understanding that $|\xi|$ and $\langle \exp_y^{-1}(x),\xi \rangle$ are taken with respect to the metric at $y$.
With the multiplicative factor $(\frac{T-\eta}{2\pi})^n$, 
\[
E_{T-\eta}(x,y) = \left( \frac{T}{2\pi} \right)^n \int_{|\xi| < 1- \frac{\eta}{T} } e^{iT \langle \exp_y^{-1}(x),\xi \rangle} \frac{d\xi}{\sqrt{|g_y|}} + O\big( T^{n-1} \big).
\]
It follows that
\[
K(x,y) = \left( \frac{T}{2\pi}\right)^n \int_{1-\frac{\eta}{T}< |\xi|_{g_y}< 1} e^{i T \langle \exp_y^{-1}(x), \xi \rangle_{g_y} } \frac{d\xi}{\sqrt{|g_y|}} + O\big(T^{n-1}\big)
\]
We assume $\eta/T \rightarrow 0$, so the integration over the thin annulus $1 - \eta/T < |\xi | < 1$ can be appproximated by an integral over the sphere $|\xi|=1$ times the width $\eta/T$. 
In polar coordinates at $y$, $g_y$ is the identity matrix. Identifying $x$ and $y$ with the vectors in $\R^n$ representing them in coordinates,
\[
\langle \exp_y^{-1}(x), \xi \rangle = (x-y) \cdot \xi + O\left( d(x,y)^2|\xi| \right).
\]
Finally, the Bessel function appears because of the integral formula
\[
\int_{|\xi|=1} e^{i (x-y) \cdot \xi } d\xi = (2\pi)^{n/2} \frac{J_{n/2-1}(|\xi|)}{|\xi|^{n/2-1}}.
\]
We summarize this as follows,
writing $d(x,y)$ for the distance between $x$ and $y$.
\begin{proposition} \label{prop:hormander}
For any compact manifold $M$ and radius $r = r(T) \rightarrow 0$ as $T \rightarrow \infty$, the two-point function for $d(x,y) < r$ is given approximately by
\[
K(x,y) =  \frac{N}{\vol(M)} \left( 2^{n/2-1}\Gamma(n/2) \frac{J_{n/2-1}(Td(x,y)) }{T^{n/2-1}d(x,y)^{n/2-1}} + O(T^{-1}) \right)
\]
where $N$
is the number of frequencies $t_j$ in the interval $T-\eta < t_j \leq T$.
We have
\[
N = c_n T^{n-1} \eta \left( 1 + O\big(T^{-1}\big) \right)
\]
for a positive constant $c_n > 0$ depending only on the dimension.
\end{proposition}
We will often use Proposition~\ref{prop:hormander} as an upper bound. The Bessel term is bounded, so we always have $K(x,y) \lesssim N \lesssim T^{n-1}\eta$. This can be improved when $d(x,y)$ is larger than $1/T$ since $J_{n/2-1}(u)$ decays as $u^{-1/2}$:
\begin{equation}
K(x,y) \lesssim T^{n/2-1/2}d(x,y)^{-n/2+1/2} \eta.
\end{equation}

To explain the normalization in Proposition~\ref{prop:hormander}, we review the local and global versions of Weyl's law. Weyl's law for counting eigenvalues can be deduced from the local version by integrating on the diagonal $x=y$.
Indeed, for orthogonal eigenfunctions $\phi_j$, we have
\[
\int_M K(x,x) dx = \sum_j 1 = N
\]
This leads to the usual Weyl's law that the number of eigenvalues in our window is
\[
N = \# \{j ; t_j \in [T - \eta(T), T] \} =\frac{|B_d| \vol(M)}{ (2\pi)^{n}} (T^{n} - (T-\eta(T))^{n}) + O(T^{n-1})
\]
where $|B_n|$ is the volume of the Euclidean unit ball in dimension $n$. Recall that we write the Laplace eigenvalue as $t_j^2$ (largely to avoid confusion with the eigenvalues of $A$, already called $\lambda_j$).
From the binomial expansion,
\begin{align*}
T^{n} - (T-\eta(T))^n &= T^{n} - (T^n - nT^{n-1}\eta(T) + \cdots )  \\
&= nT^{n-1}\eta(T) + O\big(T^{n-2}\eta(T)^2 \big) 
\end{align*}
If we choose $\eta(T) = T^{1/2}$ or smaller, then $T^{n-2}\eta(T)^2 \leq T^{n-1}$ and we can absorb the higher order terms into the error $O(T^{n-1})$ already present in Weyl's law. We need $T^{n-1}\eta(T)$ to be larger than $T^{n-1}$ in order for the main term to dominate. Thus for $\eta(T)$ diverging no faster than $T^{1/2}$, we have
\begin{equation}
\# \{j ; t_j \in [T - \eta(T) , T] \} = \frac{\vol(M) |B_n|}{(2\pi)^n} n T^{n-1} \eta(T) + O(T^{n-1})
\end{equation}
The variance of the coefficients should be inversely proportional to this in order for $\phi$ to be normalized in $L^2$ (on average). Indeed, for orthonormal eigenfunctions $\phi_j$, we have
\[
1 = \E\left[ \int_M \phi^2 \right] = \E\left[ \sum_j \sum_k c_j c_k \int_M \phi_j \phi_k \right] = \E\left[ \sum_j c_j^2 \right] =  \var[c]N.
\]

It is also of interest to study windows $\eta$ even larger than $T^{1/2}$. In this case, the error $O(T^{n-2}\eta^2)$ from the binomial expansion overwhelms the error $O(T^{n-1})$ from Weyl's law. If $\eta = (1- \alpha) T$ is proportional to $T$ instead of only $o(T)$, then there is no need to use the binomial formula and we instead have 
\[
N \sim \frac{|B_n| \vol(M)}{ (2\pi)^{n}} (1 - \alpha^n) T^n + O(T^{n-1}).
\]
This is the case of \emph{band-limited} random functions, synthesized from a band of frequencies between $(1-\alpha)T$ and $T$.

\section{Proof of Theorem~\ref{thm:2nd-mom}} \label{sec:variance}

We have Weyl's law
\[
K(x,x') = c_n \left( T^n - (T-\eta)^n \right) \left( \frac{J_{n/2-1}(Td(x,x'))}{(Td(x,x'))^{n/2-1}} + O\left( \frac{T^{n-1}}{T^n - (T-\eta)^n } \right) \right)
\]
and hence
\[
K(x,x')^2 \gtrsim \left(T^n - (T-\eta)^n\right)^2 J(x,x')^2 - O\left(T^{n-1} (T^n - (T-\eta)^n )\right).
\]
We write $J$ as a shorthand for the Bessel function 
\begin{equation} \label{eqn:jabbrev}
J = J(x,x') = J_{n/2-1}(Td(x,x'))/(Td(x,x'))^{n/2-1}.
\end{equation}
We integrate over $B \times B$ to get
\[
\int_B \int_B K(x,x')^2 dx' dx \gtrsim \left(T^n - (T-\eta)^n \right)^2 \int_B \int_B J(x,x')^2 dx' dx - O\left(\vol(B)^2 T^{2(n-1)} \eta \right)
\]
By the triangle inequality, the ball $B_r(z)$ contains a ball $B_{r-\xi}(x)$ with smaller radius and center shifted to $x$, where $\xi = d(x,z)$. The integrand $J^2$ is nonnegative, so we have a lower bound by integrating only over the smaller ball. Shifting to polar coordinates with respect to $x$, we have
\[
\int_{B_{r-\xi}(x)} J^2 \gtrsim \int_0^{r-\xi} J^2(\rho) \rho^{n-1} d\rho.
\]
Indeed, the volume form in normal coordinates is given by
\begin{equation}
d\vol(x) = (1 + O(\rho^2) ) \rho^{n-1} d\rho d\omega
\end{equation}
since the volume form is obtained from the metric $g$ by $\sqrt{\det(g)}$ and we have the expansion
\begin{equation*}
\sqrt{\det(g)} = 1 - \frac{1}{6} \text{Ric}_{kl}(x_1) x^k x^l + O(|x|^3) = 1 + O(\rho^{2} ).
\end{equation*}
We have
\[
J^2(\rho) \rho^{n-1} = J_{n/2-1}(T\rho)^2 (T \rho)^{-n+2} \rho^{n-1} = T^{-n+2} \rho J_{n/2-1}(T\rho)^2.
\]
Again, we achieve a lower bound by integrating only over a subset $B_r(z)$, namely $B_{r/2}(z)$. For $x$ in this smaller ball, we have $r - \xi \geq r/2$. Thus
\[
\int_B \int_B J^2 \gtrsim \int_{B_{r/2}(z)} \int_0^{r/2} T^{-n+2} \rho J_{n/2-1}(T\rho)^2 d\rho
\]
Change variables to $u = T\rho$ in the inner integral, so that $\rho d\rho =T^{-2} u du$. Then
\[
\int_B \int_B J^2 \gtrsim \vol(B_{r/2}(z)) T^{-n} \int_0^{rT/2} uJ_{n/2-1}(u)^2 du
\]
We have $\vol(B_{r/2}) \asymp \vol(B_r(z)) \asymp r^n$. Also, the function $uJ_{n/2-1}(u)^2$ is bounded because of the $1/\sqrt{u}$ decay of $J$-Bessel functions. It has a non-zero average value
\[
\langle uJ_{n/2-1}(u)^2 \rangle = \lim_{S \rightarrow \infty} \frac{1}{S} \int_0^S uJ_{n/2-1}(u)^2 du > 0.
\]
Hence
\begin{align*}
\int_B \int_B J^2 &\gtrsim \vol(B_r(z))^2 r^{-n} T^{-n} rT \frac{2}{rT} \int_0^{rT/2} uJ_{n/2-1}(u)^2 du \\
&\gtrsim \vol(B)^2 (rT)^{-n+1}
\end{align*}
In terms of $K$, this means
\[
\int_B \int_B K(x,x')^2 dx' dx \gtrsim (\vol(B) (T^n - (T-\eta)^n ) )^2 \left( (rT)^{-n+1} - O\left( \frac{T^{n-1}}{T^n - (T-\eta)^n} \right) \right)
\]
To have a main term, we must assume that
\[
(rT)^{-n+1} \frac{T^n - (T-\eta)^n}{T^{n-1}} \rightarrow \infty.
\]
We have
\[
T^n - (T-\eta)^n = n \eta T^{n-1} - \binom{n}{2} \eta^2 T^{n-2} + \ldots \gtrsim \eta T^{n-1}
\]
so the condition is simply
\begin{equation}
\frac{(rT)^{n-1}}{\eta} \rightarrow 0,
\end{equation}
which is one of the assumptions in Theorem~\ref{thm:main}.

One can obtain a comparable upper bound by a similar argument. Although we only need a one-sided bound to compare $\tr(A^3)$ and $\tr(A^2)$, it is worth knowing that the lower bound above is essentially of the correct order of magnitude. The ball $B_r(z)$ is contained in the re-centered ball $B_{2r}(x)$, so that
\[
\int_{B_r(z)} K(x,x')^2 dx' \leq \int_{B_{2r}(x)} K(x,x')^2 dx'.
\]
As before, we shift to polar coordinates centered at $x$. Weyl's law gives a pointwise upper bound on $K(x,x')$:
\begin{align*}
K(x,x') &\lesssim T^{n-1} \eta \left( (T\rho)^{-(n-1)/2} + \eta^{-1} \right).
\end{align*}
This diverges as $\rho \rightarrow 0$, since we would be better off using the trivial bound $K \lesssim N$ for $\rho < 1/T$, but the singularity is integrable.
We obtain
\begin{align*}
\int_{B_r(z)} K(x,x')^2 dx' &\lesssim (T^{n-1}\eta)^2 \int_0^{2r} \left( (T\rho)^{-(n-1)/2} + \eta^{-1} \right)^2 \rho^{n-1} d\rho \\
&\lesssim T^{2n-2} \eta^2 r^n \left( (rT)^{-(n-1)} + \eta^{-1} (rT)^{-(n-1)/2} + \eta^{-2} \right)
\end{align*}
Integrating over $x$ brings another factor of $\vol(B_r) \asymp r^n$. 
The resulting upper bound for the integral $\int_B \int_B K(x,x')^2 dx dx'$ is
\[
\left(\vol(B)T^{n-1}\eta\right)^2 \left( (rT)^{-(n-1)} + \eta^{-1} (rT)^{-(n-1)/2} + \eta^{-2} \right).
\]

Under our assumption, $\eta$ is large compared to $(rT)^{n-1}$ so that the terms involving $\eta^{-1}$ or $\eta^{-2}$ are even smaller than $(rT)^{-n+1}$. Therefore the upper and lower bounds match up to a constant multiple:
\begin{equation}
\int_B \int_B K(x,x')^2 dxdx' \asymp \left(\vol(B)T^{n-1}\eta \right)^2 (rT)^{-n+1}.
\end{equation}

\section{Euclidean integral for the third moment} \label{sec:euclid}

We have stated our main theorem for any compact Riemannian manifold. This level of generality is possible because the radius of the ball $B$ is shrinking and the triple integral $\int_B \int_B \int_B K_{12} K_{23} K_{31}$ is therefore close to its Euclidean counterpart.
Recall that we write $K_{12} = K(x_1,x_3), K_{23} = K(x_2,x_3), K_{31}(x_3,x_1)$ and similarly $J_{12}, J_{23} J_{31}$ for their approximations by Bessel functions.
Write Weyl's law in the form
\[
K = c_n \left( T^n - (T-\eta)^n \right)J + O(T^{n-1})
\]
where $J$ is the Bessel function from (\ref{eqn:jabbrev}).
It follows that
\[
\int_B \int_B \int_B K_{12}K_{23}K_{31} = c_n' (\eta T^{n-1})^3 \int_B \int_B \int_B J_{12}J_{23}J_{31} + O\left(  (\eta T^{n-1}\vol(B))^3\eta^{-1} \right).
\]
Write  $|x-y|$ for Euclidean distance and $d(x,y)$ for Riemannian distance between the points represented by $x$ and $y$ in coordinates. Because $J$ has bounded derivative, 
\[
J(Td(x,y)) = J(T|x-y|) + O\left( T\big(d(x,y) - |x-y| \big) \right).
\]
In a ball of radius $r \rightarrow 0$, the relative error in approximating Riemannian distances by Euclidean ones is a factor of $1+O(r)$. Hence
\[
J(Td(x,y)) = J(T|x-y|) + O(Tr|x-y|).
\]
Let $\xi = d(x_1,z)$, $\rho = d(x_1,x_3)$, $s = d(x_1,x_2)$. We compare the integral to a Euclidean version, introducing a small error in order to pass to polar coordinates based at $x_1$.
Take polar coordinates $(\rho, \omega)$ where the radial coordinate $\rho = d(x,x_1)$ ranges from 0 to $2r$. 
As in Section~\ref{sec:variance}, the volume form is given approximately by its Euclidean counterpart, namely
\begin{equation}
d\vol(x) = (1 + O(\rho^2) ) \rho^{n-1} d\rho d\omega
\end{equation}
where the error depends on the curvature of $M$.
We integrate and conclude that
\begin{proposition} \label{prop:eucl}
If $B = B_r(z)$ is a ball of radius $r \rightarrow 0$ in a compact manifold $M$ with a smooth Riemannian metric, then
\begin{align*}
&\int_B \int_B \int_B J(Td(x_1,x_2)) J(Td(x_2,x_3)) J(Td(x_3,x_1)) d\vol(x_1) d\vol(x_2) d\vol(x_3) = \\
&\int_{B_0} \int_{B_0} \int_{B_0} J(T|x_1-x_2|) J(T|x_2-x_3|) J(T|x_3-x_1|) \ dx_1 dx_2 dx_3  + O\left( \vol(B)^3 Tr^2 \right) 
\end{align*}
where $B_0$ is a Euclidean ball of radius $r$ centered at the coordinates for the basepoint $z \in M$ (or equally well, centered anywhere in $\R^n$ since the Euclidean integral is translation-invariant).
\end{proposition}
For the original kernel $K$, larger than $J$ by a factor of $T^{n-1}\eta$, this leads to an error $(\eta T^{n-1} \vol(B))^3 Tr^2$. We assume $\eta \lesssim T/(rT)^2$ so that this can be absorbed into the error term $(\eta T^{n-1} \vol(B))^3 \eta^{-1}$ already with us from Weyl's law. Thus we may proceed with the Euclidean integral.

Let us first see what upper bound follows from crude estimates ignoring cancellation in the Bessel integrand.
In higher dimensions, the angular variables $\alpha$ and $\beta$ range over $S^{n-1}$ instead of $S^1$. With $x_1$ as the origin in Euclidean space, write
\[
z = \xi \zeta, \ x_2 = s \alpha, x_3 = \rho \beta.
\]
Then
\[
|x_2 - x_3| = |s\alpha - \rho \beta|
\]
We may choose the coordinates to make $\zeta$ the north pole $(1,0,\ldots, 0)$ and then account for the integration over $\zeta$ with an overall factor of $\vol(S^{n-1})$.
The distance $\xi = d(z,x_1)$ ranges from 0 to $r$. The distances $\rho = d(x_1,x_3)$ and $s=d(x_1,x_2)$ may be as large as $r + \xi$, with an angular range starting from a full sphere $S^{n-1}$ for distances less than $r - \xi$ and shrinking to nothing as the distance approaches $r + \xi$. 
More specifically, the condition $x_2 \in B_r(z)$ is
\[
r^2 > |x_2 - z|^2 = s^2 + \xi^2 - 2\xi s \ \alpha \cdot \zeta
\]
and similarly for $x_3$ with its angle $\beta$.
Thus the angles $\alpha$ and $\beta$ obey
\begin{align} \label{align:angular-range}
\alpha \cdot \zeta &\geq \frac{s^2 + \xi^2 - r^2}{2\xi s} \\
\beta \cdot \zeta &\geq \frac{\rho^2 + \xi^2 - r^2}{2\xi \rho}
\end{align}
and we may simply take the dot products $\alpha \cdot \zeta$ and $\beta \cdot \zeta$ as the first coordinates $\alpha_1$ and $\beta_1$, if we commit to the north pole.
Let us write $\A$ for the range of $\alpha$ defined by (\ref{align:angular-range}), and likewise $\B$ for the range of $\beta$.

\begin{figure}
\begin{tikzpicture}
\draw (0,0) node{$\bullet$};
\draw (0,0.25) node {$x_1$};
\draw (0,0) circle [radius=1];
\draw (0,-0.75) node {$x_3$};
\draw (0,-1) node {$\bullet$};
\draw[dashed] (2.75,0) arc (0:75:2.75);
\draw[dashed] (2.75,0) arc (0:-75:2.75);
\draw (2.9,0.5) node {$\bullet x_2$};
\draw (2,0) node {$\bullet$};
\draw (1.75,0) node {$z$};
\draw[thick] (2,0) circle [radius=3];
\end{tikzpicture}
\caption{
An illustration of our notation for the most easily visualized case $n=2$. The thick circle bounds $B_r(z)$. The original center $z$ lies in the direction $(1,0)$ from $x_1$, which we take as a center for the calculation in polar coordinates. The distance $\rho = |x_1 - x_3|$ is small enough that $\B = S^1$, that is, the angle $\beta$ runs over a full circle. On the other hand, $|x_2 - x_1|$ is large enough that the angle $\alpha$ is restricted to a range $\A$ shown by the dashed arc.
We rescale to $w = \xi/r$, $u = Ts$, $v = T\rho$, so that $w$ measures $|x_1 - z|$ relative to the radius $r$ whereas $u$ and $v$ measure $|x_1-x_2|$ and $|x_1 - x_3|$ relative to the wave scale $1/T$.
}
\end{figure}
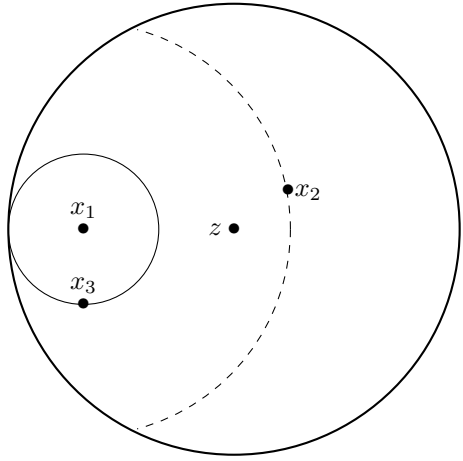

The integral is, up to a factor $\vol(S^{n-1})$,
\[
 \int_{\xi=0}^r \int_{s=0}^{r+\xi} \int_{\rho = 0}^{r+\xi} \int_{\A} \int_{\B} J(Ts)J(T\rho)J(T|s\alpha - \rho \beta |) d\alpha  d\beta \ s^{n-1} ds \ \rho^{n-1} d\rho \ \xi^{n-1} d\xi
\]
where we write $J(y) = c_n J_{n/2-1}(y)/y^{n/2-1}$ for the Bessel function. Change variables to $w = \xi/r$, $u = Ts$, $v = T\rho$, with change of measure
\[
\xi^{n-1} d\xi = r^n w^{n-1} dw, \ s^{n-1} ds = T^{-n} u^{n-1} du, \ \rho^{n-1} d\rho = T^{-n} v^{n-1} dv.
\]
The angles $\alpha$ and $\beta$ are unaffected by this scaling. The integral becomes
\begin{align*}
&\vol(S^{n-1}) r^n T^{-2n} \times \\
&\int_{w=0}^1 \int_{u=0}^{rT(1+w)} \int_{v=0}^{rT(1+w)} J(u)J(v) \int_{\A} \int_{\B} J(|u\alpha - v\beta|) d\alpha d\beta u^{n-1}du \ v^{n-1}dv \ w^{n-1}dw
\end{align*}
Think of the prefactor $r^n T^{-2n}$ as $\vol(B)^3 (rT)^{-2n}$, since $\vol(B)$ is comparable to $r^n$. 
We have
\[
J(u)u^{n-1} = J_{n/2-1}(u) u^{-(n/2-1) + n-1} = J_{n/2-1}(u) u^{n/2}
\]
which oscillates between values of order $u^{(n-1)/2}$. Simply observing that 
\[
J(|u\alpha - v\beta|) \lesssim 1
\]
and ignoring cancellation gives an upper bound
\[
\vol(B)^3 (rT)^{-2n} \left( (rT)^{(n-1)/2+1} \right)^2 = \vol(B)^3 (rT)^{-n+1}.
\]
The exponent $n-1$ must be improved to something larger than $3(n-1)/2$ in order to beat the variance, since the denominator is $\tr(A^2)^{3/2} \asymp (rT)^{-3(n-1)/2}$.

\section{Gegenbauer's addition formula} \label{sec:gegenbauer}

To handle the term $J(|u\alpha - v\beta|)$ in the inner integral, we have recourse to
\begin{fact} (Gegenbauer's addition formula, \cite[formula 10.23.8]{NIST}, \cite[p.362-368]{W})
If $\varpi = \sqrt{u^2+v^2 - 2uv \cos{\theta}}$, then for $\nu \neq 0, -1, -2, \ldots$,
\[
\frac{J_{\nu}(\varpi)}{\varpi^{\nu}} = 2^{\nu} \Gamma(\nu) \sum_{k=0}^{\infty} (\nu+k) \frac{J_{\nu+k}(u)}{u^{\nu}} \frac{J_{\nu+k}(v)}{v^{\nu}} C_k^{\nu}(\cos{\theta})
\]
where $C_k^{\nu}$ are the Gegenbauer polynomials.
\label{fact:gegenbauer}
\end{fact}
Recall that the polynomials $C_k^{\nu}(t)$ are orthogonal over the interval $[-1,1]$ with weight $(1-t^2)^{\nu - 1/2}$. Their maximum value is attained at
\begin{equation} \label{eqn:max-gegenbauer}
C_k^{\nu}(1) = \frac{(2\nu)_k}{k!} = \frac{(n-3+k)!}{(n-3)!k!} \lesssim k^{n-3}.
\end{equation}
The notation $(2\nu)_k$ is shorthand for $2\nu(2\nu+1)\cdots (2\nu+k-1)$.
Note that when $\nu =1/2$ the Gegenbauer polynomials are bounded by 1, no matter how high the degree.
For our application, $\nu = n/2-1$, so this bounded case arises when $n=3$. In higher dimensions $n > 3$, $C_k^{\nu}$ at its largest grows as a power of $k$. 
Nevertheless, the series in Fact~\ref{fact:gegenbauer} converges because of the rapid decay in $J_{\nu+k}(u)J_{\nu+k}(v)$ as a function of $k$.

To bring Gegenbauer's addition formula to bear on our inner integral, note that
\[
|u\alpha - v\beta|^2 = u^2 + v^2 - 2\alpha \cdot \beta
\]
so we may take $\theta = \arccos(\alpha \cdot \beta)$. Write $\A = \A(u,v)$ and $\B = \B(u,v)$ for the region of integration over the angles $\alpha$ and $\beta$ from (\ref{align:angular-range}), suppressing the dependence on the outermost variables $w$ and $\zeta$. The integral is
\begin{align*}
&\int_{\A(u,v)} \int_{\B(u,v)} \frac{J_{\nu}(|u\alpha - v\beta|)}{|u\alpha - v\beta|^{\nu}} \ d\alpha d\beta = \\ &\int_{\A(u,v)} \int_{\B(u,v)} 2^{\nu} \Gamma(\nu) \sum_{k=0}^{\infty} (\nu+k) \frac{J_{\nu+k}(u)}{u^{\nu}} \frac{J_{\nu+k}(v)}{v^{\nu}} C_k^{\nu}(\alpha \cdot \beta) \ d\alpha d\beta
\end{align*}
Next we integrate over $u$ and $v$. This can be interchanged with the sum over $k$, again thanks to the rapid decay of $J_{\nu+k}$ as $k$ grows. This leaves us with
\begin{align*}
 \sum_{k=0}^{\infty} (\nu+k) \int_{u=0}^{(1+w)rT} \int_{v=0}^{(1+w)rT} &du dv \ u^{n-1} v^{n-1} \frac{J_{\nu}(u)J_{\nu+k}(u)}{u^{2\nu}} \frac{J_{\nu}(v)J_{\nu+k}(v)}{v^{2\nu}} \\
&\int_{\A(u,v)} \int_{\B(u,v)} d\alpha d\beta \  C_k^{\nu}(\alpha \cdot \beta)
\end{align*}

For many ranges of the parameters, $\A(u,v)$ or $\B(u,v)$ is a full sphere $S^{n-1}$, in which case $\int_{\A} \int_{\B} C_k^{\nu}(\alpha \cdot \beta)$ is exactly 0 by orthogonality of the Gegenbauer polynomials. In general, we make the following claim.
\begin{proposition} \label{prop:ultraspherical-mono}
For $n \geq 3$, $\nu = n/2-1$, concentric arcs $\A$ and $\B$ as above, and degree $k \geq 1$
\begin{equation}
\left| \int_{\A} \int_{\B} C_k^{\nu}(\alpha \cdot \beta) \ d\alpha d\beta \right| \lesssim k^{n/2-3}.
\end{equation}
\end{proposition}

\begin{proof}
The angular ranges $\A$ and $\B$ are concentric spherical caps (centered at $\zeta$, the direction from $x_1$ to $z$). We may assume that $\A \subseteq \B$, reversing the order of $\alpha, \beta$ if need be. We integrate with respect to $\beta$ using polar coordinates $(\theta, \omega)$ based at $\alpha$. Thus $\cos{\theta} = \alpha \cdot \beta$ assumes some interval of values depending on the distance between $\alpha$ and $\zeta$, and also on the angle $\omega$. Because $\A \subseteq \B$, this interval contains $\theta = 0$ since $\beta = \alpha$ automatically lies in the range $\B$. The variable $\theta$ ranges over an interval $0 \leq \theta \leq \theta_{+}$ where $\theta_+ = \theta_+(\alpha,\omega)$, or this interval together with $\pi - \theta_+ \leq \theta \leq \pi$ if $\A$ and $\B$ exceed a hemisphere. This only changes $\cos{\theta}$ by a sign and the argument below will apply equally well. Depending on the angle $\omega$, $\theta_+$ lies somewhere between the extreme cases ${\rm diam}(\B)/2 - d(\alpha,\zeta)$ and ${\rm diam}(\B)/2 + d(\alpha,\zeta)$. 
Note also that $\theta_+ \leq \pi/2$ or else the hemispheres would merge. 
We write
\[
\int_{\B} C_k^{\nu}(\alpha \cdot \beta) d\beta = \int_{S^{n-2}} \int_0^{\theta_+} C_k^{\nu}(\cos{\theta}) \sin(\theta)^{n-2} d\theta d\omega
\]
where $(\sin\theta)^{n-2} d\theta d\omega$ is the volume element on $S^{n-1}$ and the polynomials $C_k^{\nu}(\cos\theta)$ are orthogonal on the interval $0 \leq \theta \leq \pi$ with respect to the weight $\sin(\theta)^{n-2}$. 
For the integral over $\theta$, note that with $x = \cos{\theta}$ and $\nu = n/2-1$, 
\[
C_k^{\nu}(x) (1-x^2)^{\nu-1/2} dx = C_k^{\nu}(\cos{\theta}) \sin(\theta)^{2\nu-1} \sin(\theta) d\theta = C_k^{\nu}(\cos{\theta}) \sin(\theta)^{n-2}d\theta.
\] 

We have the following exact antiderivative
\begin{equation} \label{eqn:antiderivative}
\int C_k^{\nu}(x) (1-x^2)^{\nu-1/2} dx = -\frac{1}{2k} \frac{4\nu}{2\nu+k} (1-x^2)^{\nu+1/2}C_{k-1}^{\nu+1}(x).
\end{equation}
This is a special case of formula (18.17.1) in \cite{NIST}, which gives the integral of Jacobi polynomials:
\[
\int (1-x)^a (1+x)^b P_k^{(a,b)}(x) dx = \frac{-1}{2n} (1-x)^{a+1}(1+x)^{b+1} P_{k-1}^{(a+1,b+1)}(x).
\]
The conversion between the ultraspherical notation $C_k^{\nu}$ and the Jacobi polynomial is
\[
C_k^{\nu}(x) = \frac{(2\nu)_k}{(\nu+1/2)_k} P_k^{(\nu-1/2,\nu-1/2)}(x)
\]
where $(t)_k = t(t+1)\cdots (t+k-1)$ denotes a rising factorial.
We have
\begin{align*}
\frac{(2\nu)_k}{(2\nu+2)_{k-1}} &= \frac{2\nu (2\nu+1) \ldots (2\nu+k-1) }{(2\nu+2)(2\nu+3) \ldots (2\nu+2+k-2) } = \frac{2\nu}{2\nu +k} \\
\frac{(\nu+3/2)_{k-1} }{(\nu+1/2)_k } &= \frac{(\nu+3/2)\ldots(\nu+3/2+k-2)}{(\nu+1/2)\ldots (\nu+1/2+k-1)} = \frac{1}{\nu+1/2}
\end{align*}
Converting between $C_k^{\nu}$ and $P_k^{(\nu-1/2,\nu-1/2)}$ under the integral, as well as $P_{k-1}^{(\nu+1/2,\nu+1/2)}$ and $C_{k-1}^{\nu+1}$ outside it, therefore brings a factor
\begin{equation*}
\frac{(2\nu)_k}{(\nu+1/2)_k} \frac{(\nu+3/2)_{k-1}}{(2\nu+2)_{k-1} } = \frac{2\nu}{2\nu+k} \frac{1}{\nu+1/2} = \frac{4\nu}{2\nu+k}
\end{equation*}
as stated in (\ref{eqn:antiderivative}).

For the definite integral, note that when $\theta = 0$, we have $1-x^2 = 0$, so that only the upper endpoint $\theta_+$ contributes.
It follows that
\[
\int_0^{\theta_+} C_k^{\nu}(\cos{\theta}) (\sin\theta)^{n-2} d\theta = \frac{2\nu}{k(2\nu+k)} (\sin\theta_+)^{2\nu+1}C_{k-1}^{\nu+1}(\cos{\theta_+}).
\]

We recall Szeg\H{o}'s asymptotic formula (8.21.17) in \cite{S}: For large degree $k$, the Gegenbauer polynomial is given approximately by
\[
\frac{C_k^{\nu}(\cos{\theta}) }{C_k^{\nu}(1) } (\sin\theta)^{n-2} \approx \frac{J_{\nu-1/2}((k+\nu)\theta)}{((k+\nu)\theta)^{\nu-1/2}  }\left(\lim_{t \rightarrow 0} J_{\nu-1/2}(t)/t^{\nu-1/2} \right)^{-1} (\sin\theta)^{n-2}
\]
up to an additive error of $O(\theta^{1/2}k^{-\nu-1})$, and even smaller if $\theta \lesssim 1/k$.
In particular, applying this with $\nu+1$ in place of $\nu$ gives
\[
\left| \sin(\theta_+)^{2\nu+1}C_{k-1}^{\nu+1} (\cos{\theta_+}) \right| \lesssim C_{k-1}^{\nu+1}(1) \left| \frac{J_{\nu+1/2}((k+\nu)\theta_+)}{((k+\nu)\theta_+)^{\nu+1/2}} \right| \theta_+^{2\nu+1}.
\]
We have already noted that $C_k^{\nu}(1) \lesssim k^{n-3}$, and changing $\nu = n/2-1$ to $\nu+1$ corresponds to increasing $n$ by 2. Thus 
\[
C_{k-1}^{\nu+1}(1) \lesssim (k-1)^{(n+2)-3} \leq k^{n-1}.
\] 
The large-argument bound $J_{\nu+1/2}(y) \lesssim y^{-1/2}$ then gives
\[
C_{k-1}^{\nu+1}(1) \left| \frac{J_{\nu+1/2}((k+\nu)\theta_+)}{((k+\nu)\theta_+)^{\nu+1/2}} \right| \theta_+^{2\nu+1} \lesssim k^{n-1} (k\theta_+)^{-\nu-1} \theta_+^{2\nu+1} \lesssim k^{n/2-1}.
\]
We have used the fact that $n \geq 3$ to bound $\theta_+^{\nu} \lesssim 1$, the exponent being nonnegative.
Taking account of the prefactor 
\[
\frac{2\nu}{k(k+2\nu)} \lesssim k^{-2},
\]
it follows that
\[
\int_0^{\theta_+} C_k^{\nu}(\cos{\theta}) \sin(\theta)^{n-2} d\theta \lesssim k^{-2}k^{n/2-1} = k^{n/2-3}.
\]
Integrating over $\alpha$ and $\omega$ brings only a bounded factor (volumes of spheres, depending only on $n$) and the claim follows. 
\end{proof}

Using Proposition~\ref{prop:ultraspherical-mono}, the upper bound for the third moment becomes
\begin{equation} \label{eqn:gegenbauer-bound}
\vol(B)^3 (rT)^{-2n} \sum_k (\nu+k) k^{n/2-3} \left( \int_0^{rT} uJ_{\nu}(u)J_{\nu+k}(u) \ du \right)^2.
\end{equation}

\section{Beginning, middle, and end} \label{sec:bessel}

In this section, we estimate the following one-variable integral arising from Section~\ref{sec:gegenbauer}.
\begin{proposition} \label{prop:bessel-int-1var}
As $rT \rightarrow \infty$ with $m \leq 2rT$,
\begin{equation} \label{eqn:small-m}
\int_0^{(1+w)rT} u|J_{\nu}(u) J_m(u)| du \lesssim rT
\end{equation}
whereas if $m > 2rT$, then
\begin{equation} \label{eqn:large-m}
\int_0^{(1+w)rT} u|J_{\nu}(u)J_m(u)| du \leq \exp(-c(b) m^{b})
\end{equation}
for any $b<1$, with $c(b) > 0$ a correspondingly small constant. 
Both cases (\ref{eqn:small-m}) and (\ref{eqn:large-m}) are uniform with respect to $0 \leq w \leq 1$ (or any bounded interval of $w$).
\end{proposition}

The proposition is easy to see for fixed $m$, because the integrand $uJ_{\nu}(u)J_m(u)$ is a bounded function of $u$ and the interval of integration has length at most $2rT$. The difficulty lies in allowing $m$ to grow with $rT$, but we will see that the same reasoning applies after subdividing the integral depending on the size of $m$.
For our application, $m = k+\nu$ as before.

Fix a $\delta > 0$ and let $m_{\pm} = m \pm m^{1/3+\delta}$. Then split the integral as
\[
\int_0^{(1+w)rT} u |J_{\nu}(u)J_m(u)| = \int_0^{m_-} + \int_{m_-}^{m_+} + \int_{m_+}^{(1+w)rT}
\]
where the latter two ranges are empty if $m$ is large enough that $m_- \geq (1+w)rT$.
These ranges are chosen in light of the fact that $J_m(u)$ is very small when $u$ is much less than $m$, mildly decaying and oscillatory when $u$ is much larger than $m$, and with a transition between these regimes taking place when $|u-m| \lesssim m^{1/3}$. More precisely,
\begin{proposition} \label{prop:bessel-facts}
For $0 < x \leq 1$ and $m \geq 0$ not necessarily an integer,
\begin{equation} \label{eqn:kap}
\log |J_m(mx)| \leq m \left(\log{x} + \sqrt{1-x^2} - \log(1+\sqrt{1-x^2}) \right).
\end{equation}
For $u > m$,
\begin{equation} \label{eqn:large-argument}
|J_m(u)| \lesssim (u^2-m^2)^{-1/4}.
\end{equation}
For any $u$, but of particular importance when $|u-m| < m^{1/3}$,
\begin{equation} \label{eqn:transition}
J_m(u) \lesssim m^{-1/3}.
\end{equation}
\end{proposition}
The first of these is inequality (10.14.15) in \cite{NIST}. The latter two bounds are crude consequences of much more detailed asymptotics available for $J_m(u)$. See, for instance, (10.20.4) of \cite{NIST} or Chapter 7.4 of \cite{E} for statements. For proofs, see chapter 8 of Watson's book \cite{W} or Olver's article \cite{O}. Note that the bounds are compatible in the sense that (\ref{eqn:large-argument}) recovers (\ref{eqn:transition}) when $u-m$ is of order $m^{1/3}$. Note also that (\ref{eqn:large-argument}) reduces to $J_m(u) \lesssim u^{-1/2}$ when $u$ is much larger than $m$, for instance if $m$ is fixed.

In the initial range $\int_0^{m_-}$, we apply (\ref{eqn:kap}) with $x \leq 1 - m^{-2/3+\delta}$. First, we use 
\[
\log{x} = \frac{1}{2}\log(1-{\sqrt{1-x^2}}^2)
\]
to expand the upper bound
\[
\log |J_m(mx)| \leq m \left(\log{x} + \sqrt{1-x^2} - \log(1+\sqrt{1-x^2}) \right)
\]
as a power series in $\sqrt{1-x^2}$.
Note the series expansion
\[
\frac{1}{2} \log(1-y^2) + y - \log(1+y) = -\frac{1}{3}y^3 - \frac{1}{5}y^5 - \ldots
\]
with $y = \sqrt{1-x^2}$ for our application.
It follows that
\[
\log{x} + \sqrt{1-x^2} - \log(1+\sqrt{1-x^2}) \leq -\frac{1}{3} (1-x^2)^{3/2} .
\]
In the range $u \leq m_-$, we have $x \leq 1 - m^{-2/3+\delta}$. Thus
\[
\frac{-m}{3}(1-x^2)^{3/2} \leq -\frac{m^{2/3+\delta/2}}{3} (1-x).
\]
Now we can use $J_{\nu}(u) \lesssim u^{-1/2}$ to bound the initial contribution by an elementary integral:
\begin{align*}
\int_1^{m_-} u|J_{\nu}(u)J_m(u)| du &\lesssim m\int_{1/m}^{1-m^{-2/3+\delta}} (mx)^{1/2} e^{m(\log{x} + \sqrt{1-x^2} - \log(1+\sqrt{1-x^2}))} dx \\
&\lesssim m\int_0^{1-m^{-2/3+\delta}} \exp\left( - \frac{m^{2/3+\delta/2}}{3}(1-x) \right) dx \\
&\lesssim \exp\left(-cm^{3\delta/2} \right)
\end{align*}
for a constant $c > 0$.

For the middle contribution, we use the transitional bound $J_m(u) \lesssim m^{-1/3}$ together with $J_{\nu}(u) \lesssim u^{-1/2}$ for fixed $\nu$ to conclude that
\begin{align*}
\int_{m_-}^{m_+} u|J_m(u)J_{\nu}(u)| du &\lesssim m^{-1/3} (m_+ - m_-) m^{1/2} \\
&= 2m^{1/2 + \delta}
\end{align*}

For the final contribution, we see from (\ref{eqn:large-argument}) that
\[
\int_{m_+}^{(1+w)rT} |uJ_{\nu}(u)J_m(u)| du \lesssim \int_{m_+}^{2rT} u^{1/2} (u^2-m^2)^{-1/4}du.
\]
Changing variables to $x = u/m$ gives
\[
\int_0^{(1+w)rT} |uJ_{\nu}(u)J_m(u)| du \lesssim m \int_{1 + m^{-2/3+\delta}}^{2rT/m} x^{1/2} (x^2-1)^{-1/4} dx
\]
Because $x^{1/2}(x^2-1)^{-1/4}$ is integrable near $x=1$ and bounded for $x$ away from 1, we have
\begin{align*}
m \int_{1 + m^{-2/3+\delta}}^{2rT/m} x^{1/2} (x^2-1)^{-1/4} dx &\lesssim m\int_1^{2rT/m} (1-x^{-2})^{-1/4} \\
&\lesssim m(1 + rT/m)
\end{align*}
If $m > 2rT$, then for any value of $w$, only the initial segment $\int_0^{m_-}$ is present and so the integral is exponentially small, as claimed in (\ref{eqn:large-m}). If $m \leq 2rT$, then in principle there may also be contributions as large as $m(1+rT/m)$ and $2m^{1/2+\delta}$, both of which are dominated by $rT$ for this range of $m$. Thus the integral is of the order $rT$ as claimed in (\ref{eqn:small-m}). This proves Proposition~\ref{prop:bessel-int-1var}.

\section{Completing the proof} \label{sec:complete}

In Section~\ref{sec:gegenbauer}, we used Gegenbauer's addition formula to express the third moment as a sum obeying the upper bound
\begin{equation*}
\vol(B_r)^3  (rT)^{-2n} \sum_{k=0}^{\infty} (\nu+k) k^{n/2-3} \left( \int_0^{(1+w)rT} uJ_{\nu}(u)J_{\nu+k}(u) du \right)^2 
\end{equation*}
up to a constant multiple.
In Section~\ref{sec:bessel}, we estimated the integral appearing in the summands for each $k \geq 0$.
In this section, we estimate the sum over $k$ and deduce that the third moment compares favourably to the variance. We write $m = \nu+k$.

We distinguish a bulk contribution, where $m \leq 2rT$, from a tail contribution where $m > 2rT$. 
For $m$ in the bulk, combining the three ranges of integration gives the upper bound from (\ref{eqn:small-m}):
\[
\int_0^{(1+w)rT} u|J_{\nu}(u)J_m(u)| du \lesssim rT + m^{1/2+\delta} + e^{-cm^{3\delta/2}} \lesssim rT.
\]
For $m$ in the tail, only the beginning of the integral is present:
\[
\int_0^{(1+w)rT} u|J_{\nu}(u)J_m(u)| du \lesssim  e^{-cm^{3\delta/2}}.
\]
Since $m^{n/2-5/2}\exp(-cm^{3\delta/2})$ is summble, the tail contributes only a bounded amount in addition to the bulk sum:
\[
\sum_{m=\nu}^{\infty} m^{n/2-52} \left( \int_0^{(1+w)rT} u J_{\nu}(u)J_m(u) du \right)^2 \lesssim 1 + \sum_{m < 2rT} m^{n/2-5/2} (rT)^2
\]
The upper bound on the third moment becomes
\begin{align*}
&\vol(B_r)^3  (rT)^{-2n} \sum_{k=1}^{\infty} (\nu +k)k^{n/2-3} \left( \int_0^{(1+w)rT} uJ_{\nu}(u)J_{\nu+k}(u) du \right)^2 \\
&\lesssim \vol(B)^3 (rT)^{-2n} \sum_{m < 2rT} m^{n/2-2} (rT)^2 \\
&\lesssim \vol(B)^3 (rT)^{-2n+2+n/2-2+1} \\
&= \vol(B)^3 (rT)^{-3n/2+1}.
\end{align*}
This compares favourably to the variance:
\[
\frac{\int_B \int_B \int_B J_{12}J_{23}J_{31} }{\left( \int \int J_{12}^2 \right)^{3/2} } \lesssim \frac{(rT)^{-3n/2+1}}{(rT)^{-3(n-1)/2}} =  (rT)^{-1/2} \rightarrow 0
\]
which converges to $0$.

For any (positive definite) quadratic form in Gaussians, Proposition~\ref{prop:quads} gives
\[
\left| \sum_{p=3}^{\infty} \tr(A^p) \frac{2^{p-1} i^p t^p}{p\sigma^p} \right| \leq \sum_{p=3}^{\infty} \left( \frac{ 2|t|\tr\big( A^3 \big)^{1/3} }{\big(2 \tr(A^2) \big)^{1/2}}\right)^p \lesssim |t|^3 \frac{ \tr(A^3) }{\tr(A^2)^{3/2}}
\]
We have shown that for our particular form $\int_B \phi^2$ in any dimension $n \geq 3$, the second and third moments obey
\begin{align*}
\frac{\tr(A^3)}{\tr(A^2)^{3/2}} &= \frac{ \int_B \int_B \int_B K(x_1,x_2)K(x_2,x_3)K(x_3,x_1) dx_3 dx_2 dx_1 }{\left( \int_B \int_B K(x,x')^2 dx dx' \right)^{3/2}} \\
&\lesssim (rT)^{-1/2}
\end{align*}
This implies the central limit theorem in the form
\[
\E\left[ e^{i t Z} \right] = e^{-t^2/2} \left(1 + O\left( (rT)^{-1/2} \right) \right)
\]
as $rT \rightarrow \infty$. 
Thus we have proved the central limit theorem for $n \geq 3$ using Gegenbauer's addition formula.

\section{Two-dimensional case} \label{sec:2d}

In dimension 2, we have $\nu = n/2-1 = 0$ so that Gegenbauer's addition formula is not available and we must argue differently.
Note that the variance bound is the same for $n=2$ as for higher $n$, namely $(rT)^{-n+1}$ or just $(rT)^{-1}$ in this case. A significant difference is that the exponent $\nu-1/2 = (n-3)/2$ on the weight $(1-t^2)^{\nu-1/2}$ becomes negative. The kernel $J_0(u)$ decays only as $u^{-1/2}$ instead of saving an additional power $u^{-(n/2-1)}$. The Gegenbauer polynomials reduce to Chebyshev polynomials $T_n(\cos{\theta}) = \cos(n\theta)$. 

As $r \rightarrow 0$, the local Weyl law leads to the same triple integral in Section~\ref{sec:euclid} no matter the global geometry of $M$. As a result, if a direct calculation proves the limit theorem on any particular manifold, one can draw the same conclusion in general. Such a calculation is possible on the sphere, and we will deduce the two-dimensional case from the instance $M = S^2$. First we review the theory of spherical harmonics, including the case of higher dimensions. In any dimension, one could in principle use the same strategy of reducing to the sphere, but the two-dimensional case is simplest.

Note that the window $\eta$ plays a minor role. Under the conditions of Theorem~\ref{thm:main}, the central limit theorem reduces to bounding the Euclidean multiple integrals $\int \int \int K_{12}K_{23}K_{31}$ and $\int \int K^2$. These same integrals are obtained from the sphere with $\eta$ small enough that the window $T-\eta < t_j \leq T$ amounts to a single degree of spherical harmonics. Thus the general case of any manifold $M$ with a window as in Theorem~\ref{thm:main} follows from the case of $M = S^2$ with, say, $\eta = 1/2$.

\subsection{Spherical harmonics} \label{sec:spharmonics}

The standard basis of spherical harmonics of degree $l$ on $S^{p+1}$ is parametrized by integers 
\[
m_0 \geq m_1 \geq \ldots \geq m_p \geq 0
\]
whose sum is the total degree $\sum_j m_j = l$. For any such vector $m = (m_0, \ldots, m_p, \pm)$ together with a choice of sign $\pm$, one has a spherical harmonic
\[
f_{m} = e^{\pm i m_p \varphi} \prod_{k=0}^{p-1} (\sin\theta_{k+1})^{m_{k+1}} C_{m_k - m_{k+1}}^{m_{k+1}+(p-k)/2}(\cos {\theta_{k+1}} )
\]
where $C_{n}^{\lambda}$ denotes the Gegenbauer polynomial.
 As $m$ varies, these harmonics are orthonormal and span the space of all harmonics of degree $l$. For real-valued harmonics, it suffices to replace $e^{\pm i m_p \varphi}$ by $\sin(m_p \varphi)$ and $\cos(m_p \varphi)$ (up to a $\sqrt{2}$ factor of normalization).
These functions come from separation of variables in the coordinates
\begin{align*}
x_1 &= \cos{\theta_1} \\
x_2 &= \sin{\theta_1} \cos{\theta_2} \\
&\vdots \\
x_{p+1} &= \sin\theta_1 \cos \theta_2 \ldots \cos\varphi \\
x_{p+2} &= \sin\theta_1\cos\theta_2 \ldots \sin\varphi
\end{align*}
on $\R^{p+2}$, where the angles $\theta_j$ range from $0$ to $\pi$ while the final angle $\varphi$ ranges over a full circle $0 \leq \varphi < 2\pi$. We refer to \cite{E}, section 11.2 and especially formula 11.2(23).

Fix any point of $S^{p+1}$ as origin, say $(1,0,\ldots, 0)$. The distance to the origin is then given by $\cos{\theta_1}$. We separate variables into a radial part and an angular part varying over the lower-dimensional sphere $S^p$. Thus
\[
f_{m}(\omega) = R_{m}(\theta_1)Y_{m}(\alpha)
\]
where
\[
R_{m}(\theta_1) =  (\sin\theta_{1})^{m_{1}} C_{m_0- m_{1}}^{m_{1}+p/2}(\cos {\theta_{1}} )
\]
and the remaining factors are such that $Y_{m}$ is a spherical harmonic of degree $l-m_0$ on $S^p$. Note that 
\[
x_2^2 + \ldots x_{p+2}^2 = 1 - x_1^2 = (\sin\theta_1)^2
\]
and the corresponding coordinates on the lower-dimensional sphere $S^{p}$ are simply scaled by $\sin{\theta_1}$ in order to get $x_2, \ldots x_{p+2}$. The spherical volume element on $S^{p+1}$ is
\begin{align*}
d\omega &= (\sin\theta_1)^{p} (\sin\theta_2)^{p-1} \cdots (\sin\theta_p) \ d\theta_1 d\theta_2 \ldots d\theta_p d\varphi \\
&= (\sin\theta_1)^{p} d\theta_1 d\alpha
\end{align*}
where $d\alpha$ is the spherical volume element on $S^p$. In particular, two harmonics $f_m$ and $f_{m'}$ with vectors $m \neq m'$ are orthogonal on any ball $\{ \theta_1 < r \}$ because the angular parts are already orthogonal over $S^p$. Note that the condition $\sum_j m_j = l$ rules out the possibility that $f_m$ and $f_{m'}$ agree on their angular coefficients but have $m_0 \neq m'_0$. Such harmonics would be of different degree.

For the two-dimensional sphere $S^2$, we have $p=1$. The coordinates are just $\theta_1$ and $\varphi$. Because of the constraint $m_0+m_1 = l$, there is essentially just one component $m_0$ together with a choice of sign $\pm$. Thus we have the usual basis harmonics $Y_l^m$
\[
e^{\pm (l-m_0)\varphi} (\sin\theta_1)^{m_0} C_{2m_0-l}^{l-m_0+1/2}(\cos{\theta_1})
\]

Now expand any given harmonic $\phi$ in the basis functions $f_m$ with respect to a chosen origin:
\[
f = \sum_m c_m f_m
\]
where $m=(m_0,\ldots,m_{p}, \pm)$ varies over all vectors as above, that is, $\sum m_j = l$ and the entries of $m$ are decreasing from $m_0$ to $m_p$. 
We also specify a trigonometric factor $e^{\pm i m_p \varphi}$ via the choice of sign.
By orthogonality, integrating $f^2$ over a ball around the origin gives a diagonal quadratic form in the coefficients $c_m$:
\[
\int_B f^2 = \sum_m c_m^2 \int_B f_m^2.
\]
For our application, the functions $f_m$ must be normalized so that $\int_{S^{p+1}} f_m^2 = 1$. Thus
\[
\int_B f^2 = \sum_m c_m^2 \frac{ \int_0^r R_m(\theta_1)^2 \sin(\theta_1)^p d\theta_1 }{\int_0^{\pi} R_m(\theta_1)^2 \sin(\theta_1)^p d\theta_1 }
\]
since the angular integrals over $S^p$ are the same whether we integrate over the ball of radius $r$ or over the whole sphere $S^p$, that is, the ball of radius $\pi$. Note that the radial part $R_m$ depends only on $m_0$ and $m_1$. Given $m_0$ and $m_1$, the extra components $m_2, \ldots, m_p$ and the sign $\pm$ weight the corresponding summand by a combinatorial factor.
For example, if $m_0+m_1 = l$, the weight is 2 because it must be that $m_2 = \ldots m_p = 0$, leaving only the two choices of sign. Likewise, if $m_0+m_1 = l-1$, it must be that $m_2 = 1$ and the smaller components vanish, so the weight is again 2. But the larger the deficit between $m_0+m_1$ and the total degree $l$, the more combinations arise. A simplification available for $S^2$ compared to higher-dimensional spheres is that $m_0,m_1$ are the only parameters and no further cases occur.

For large $n$, we have Szeg\H{o}'s asymptotic for the Gegenbauer polynomial:
\[
\left( \frac{1}{2}\sin{\theta} \right)^{a} P_n^{(a,a)}(\cos{\theta}) \approx N^{-a} \frac{\Gamma(n+a+1)}{n!} J_{a}(N\theta)
\]
where $N = n + a + 1/2$. The error term is $O(\theta^{1/2}n^{-3/2})$, and even smaller if $\theta \lesssim 1/n$.
See formula (8.21.17) in \cite{S}.
For us, $n = m_0-m_1$ and $a = m_1 + (p-1)/2$.
Thus
\[
\frac{ \int_0^r C_{m_0-m_1}^{m_1+p/2}(\cos{\theta}) \sin(\theta)^{2m_1+p} d\theta }{\int_0^{\pi} C_{m_0-m_1}^{m_1+p/2}(\cos{\theta}) \sin(\theta)^{2m_1+p} d\theta } \approx \frac{\int_0^{r(m_0+p/2)} tJ_{m_1+p/2}(t)^2 dt}{\int_0^{\pi(m_0+1/2)} tJ_{m_1+p/2}(t)^2 dt}
\]

\subsection{Central limit theorem on the two-dimensional sphere} \label{sec:s2}

When $M = S^2$, the Laplace eigenvalues are $m(m+1)$ with multiplicity $2m+1$, where $m = 0, 1, 2, \ldots$ is the degree (instead of $l$, as above in arbitrary dimension). If $T = \sqrt{m(m+1)}$ and $\eta$ is a small constant (for instance, $1/2$), then the window $T - \eta < t_j \leq T$ contains $2m+1$ copies of a single frequency equal to $\sqrt{m(m+1)}$.
Note that $T = \sqrt{m(m+1)} \sim m+1/2$, so we may take $m$ rather than $T$ as the key asymptotic parameter.
The eigenvalues of the quadratic form $\int_B \phi^2$ are given by
\begin{equation*}
\lambda_j = \frac{1}{2m+1} \frac{1}{\vol(B_r)} \int_{B_r} \phi_j^2.
\end{equation*}
This was used in \cite{dci} to estimate the $\lambda_j$ as follows. There are three cases, which we think of as ``bulk", ``edge", and ``tail".

\begin{proposition} \label{prop:diag}
Fix a point $z \in S^2$ and choose for our basis functions $\phi_j$ the standard ultraspherical polynomials rotated so that $z$ is at the North pole $(0,0,1)$. Then
\begin{equation*}
\frac{1}{\vol(B_r(z)} \int_{B_r(z)} \phi^2 = \sum_{\nu = 0}^{2m} \lambda_{\nu} \mathfrak{z}_{\nu}^2
\end{equation*}
where the $\mathfrak{z}_{\nu}$ are independent standard Gaussians for $0 \leq \nu \leq 2m$ and the coefficients $\lambda_{\nu}$ satisfy
\begin{equation}
\lambda_{2k} = \lambda_{2k+1} = \frac{1}{2 \pi^2} \sqrt{1 - \left(\frac{k}{rm}\right)^2 }\frac{1}{rm}\left(1 + O_{a}\left(\frac{k^{2/3+a}}{rm}\right) \right)
\end{equation}
for any $a > 0$ and a ratio $0 \leq k/(rm) < 1$ bounded away from $1$. For $k$ close to $rm$, we have an upper bound
\begin{equation}
\lambda_k \lesssim_{a} (rm)^{-4/3 + a}.
\end{equation}
For $k$ so large that $k + k^p > rm$, where $p > 1/3$, we have
\begin{equation}
\lambda_k \lesssim_p \frac{ \exp(-c k^{(3p-1)/2}) }{(rm)^2}
\end{equation}
for a constant $c > 0$.
\end{proposition}

\begin{figure}[h]
\includegraphics[width=0.5\textwidth]{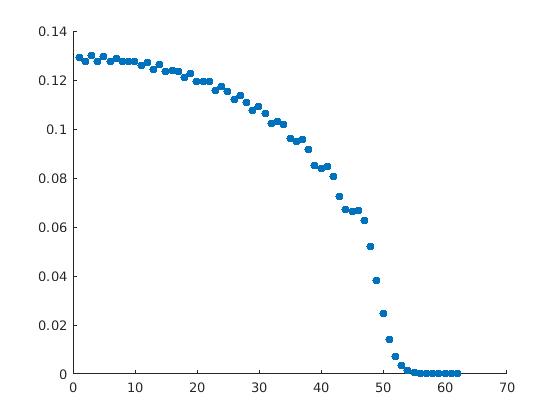}
\label{fig:semicircle}
\caption{
Ultraspherical integrals $\lambda_0, \lambda_1, \lambda_2, \ldots, \lambda_{2m}$ of degree $m=128$ for radius satisfying $rm = (\log{m})^{5/2}$. The plot shows that $\lambda_k$ approximately follows a semicircle and cuts off sharply near $rm = 51.857\ldots$, consistent with the support in Proposition~\ref{prop:diag}
}
\end{figure}

Thus the eigenvalues of $A$ follow the semi-circle law. Perhaps one should expect some GOE behaviour because $A$ is given by integrals $\int_B \phi_j \phi_k$ and there is a large orthogonal group of symmetries acting by change of basis on the functions $\phi_j$.

As in Section~\ref{sec:quad}, to show that $Z$ is approximately Gaussian, it is sufficient to show that
for $p \geq 3$
\begin{equation} \label{eq:himom}
\frac{\sum \lambda_j^p}{\sigma^p} \rightarrow 0.
\end{equation}
To verify (\ref{eq:himom}), we can use the semicircle law from Proposition \ref{prop:diag} to estimate $\sum \lambda_j^p$. First, we bound the contributions from $j$ near $rm$ or larger. 
Fix a value of $\delta > 0$ and write $j \approx rm$ as an abbreviation for $|j-rm| \leq (rm)^{1/3+\delta}$.
We have
\[
\sum_{j \approx rm} \lambda_j^p \lesssim_a (rm)^{1/3} (rm)^{p(-4/3+a)}.
\]
For larger $j$, we have
\[
\sum_{j + j^{1/3+\delta} > rm} \lambda_j^p \lesssim \frac{1}{(rm)^2} \sum_{k > rm} \exp(-ck^{3\delta/2})
\]
The tail sum is very small. For example, comparing to an integral gives
\[
\sum_{x > L} e^{-x^{\varepsilon}} \lesssim \int_L^{\infty} e^{-x^{\varepsilon}} dx \lesssim \frac{1}{\varepsilon} e^{-L^{\varepsilon}} L^{1-\varepsilon}
\]
with $\varepsilon = 3\delta/2$, $x \approx c^{2/(3\delta)}k$ so that $x^{\varepsilon} \approx ck^{3\delta/2}$, and $L = c^{1/\varepsilon} rm$. Meanwhile, the ``bulk" contribution is
\begin{align*}
\sum_{j < rm - (rm)^{1/3+\delta}} \lambda_j^p &= 2 \sum_{k=1}^{rm} \left( \frac{1}{2 \pi^2} \sqrt{1 - (k/(rm))^2 }\frac{1}{rm}\left(1 + O_{a}\left(\frac{k^{2/3+a}}{rm}\right) \right) \right)^p \\
&= \frac{2}{(2\pi^2)^p} \frac{1}{rm} \sum_{k=1}^{rm} \left( 1 - \left(\frac{k}{rm}\right)^2 \right)^{p/2} \frac{rm}{(rm)^p} \left(1 + O_{a}\left( \frac{p k^{2/3 + a}}{rm} \right) \right) 
\end{align*}
For the main term, we have a Riemann sum approximation
\[
\frac{1}{rm} \sum_{k=1}^{rm} (1- (k/(rm))^2 )^{p/2} \sim \int_0^1 (1-u^2)^{p/2} du.
\]
For the error term, we have
\begin{align*}
&\frac{2p}{(2\pi^2)^p(rm)^{p+1}} \sum_{k=1}^{rm} \left( 1 - \left(\frac{k}{rm}\right)^2 \right)^{p/2} k^{2/3+a} \\
& \lesssim \frac{p}{(2\pi^2)^p}\frac{(rm)^{2/3+a}}{(rm)^p} \int_0^1 (1-u^2)^{p/2} u^{2/3+a} du.
\end{align*}
Thus the bulk contribution is
\begin{equation}
\sum_{j \in \text{bulk}} \lambda_j^p = \frac{2}{(2\pi^2)^p} (rm)^{-p+1} \left( \int_0^1 (1-u^2)^{p/2} du + O_{p,a}\big((rm)^{-1/3 + a}\big) \right).
\end{equation}
Compare this to the edge contribution, bounded by $(rm)^{(-4/3+a)p+1/3}$, and the tail contribution, which obeys the even stronger bound $\exp(-c(rm)^{3\delta/2})$. For any chosen $a < 1/3$, these are negligible compared to the main term of order $(rm)^{-p+1}$ and even to the error term of order $(rm)^{-p+1-1/3+a}$. Thus we have, for the sum over all three ranges,
\[
\sum_j \lambda_j^p = \frac{2}{(2\pi^2)^p} (rm)^{-p+1} \left( \int_0^1 (1-u^2)^{p/2} du + O_{p,a}\big((rm)^{-1/3 + a}\big) \right)
\]
with an implicit constant proportional to $p\int_0^1 (1-u^2)^{p/2} u^{2/3+a} du$. 
In particular,
\begin{equation}
\sigma^2 = 2\sum_j \lambda_j^2 \sim \frac{1}{\pi^4} \frac{1}{rm}.
\end{equation}
From this, we obtain
\begin{equation}
\frac{\sum \lambda_j^p}{\sigma^p} \lesssim \frac{ (rm)^{-p+1} }{(rm)^{-p/2} } = (rm)^{-p/2+1},
\end{equation}
which is negligible as $rm \rightarrow \infty$ provided $p \geq 3$. Even when we sum over all $p \geq 3$, we obtain a geometric series:
\[
\sum_{p=3}^{\infty} \frac{\sum \lambda_j^p}{\sigma^p} \lesssim rm \sum_{p=3}^{\infty} \sqrt{rm}^{-p} \lesssim (rm)^{-1/2}.
\]

A final detail remains: We took $s= it/\sigma$, which we now justify by using the semicircle law to show that the power series for $\log(1-2s\lambda_j)$ does converge at that point. From the semicircle law Proposition~\ref{prop:diag}, the largest $\lambda_j$ are of order $1/(rm)$, whereas $\sigma$ is of order $1/\sqrt{rm}$. It follows that, for any given $t \in \R$, once $rm$ is sufficiently large we do have
\[
2 |t| \frac{\lambda_j}{\sigma} < 1
\]
for all $j$. The series will converge once $rm$ is so large that
\begin{equation}
\frac{1}{\sqrt{rm}} < \frac{1}{2|t|}.
\end{equation}
Combining these estimates, we have
\[
\log\E[e^{itZ}] = \frac{-t^2}{2} + \sum_{p=3}^{\infty}  \left( \frac{\sum_j \lambda_j^p}{\sigma^p} \right) (2it)^p = -\frac{t^2}{2} + O\left( (rm)^{-1/2} |t|^3 \right)
\]
and hence
\[
\E\left[ e^{itZ} \right] = \left(1 + O\big( (rm)^{-1/2}|t|^3 \big) \right) e^{-t^2/2}.
\]
For any fixed $t$, this converges to $e^{-t^2/2}$ as $rm \rightarrow \infty$, which completes our proof of Theorem~\ref{thm:main}.

\subsection{Failure of the CLT at the wave scale} \label{sec:fail}
Suppose that $rT = c$ for a fixed constant $c$, or more generally that $rT$ remains bounded as $T \rightarrow \infty$. At this scale, we do not expect the central limit theorem to hold, and a direct calculation on $S^2$ supports this. From the variance formula (\ref{eq:p-moment}) with $p=2$,
\[
\sum_j \lambda_j^2 = \tr(A^2) = \int_B \int_B \left( \frac{K(x,y)}{N} \right)^2 \frac{dx}{\vol{B}} \frac{dy}{\vol{B}} \lesssim 1. 
\]
We have $\lambda_j > 0$ because they are the eigenvalues of the positive quadratic form $\int_B \phi^2$.
It follows that
\[
\sum_j \lambda_j^3 \geq \lambda_0^3
\]
where $\lambda_0$ is the eigenvalue corresponding to the zonal spherical harmonic.
As in the proof of the semicircle law Proposition~\ref{prop:diag}, we have
\[
\lambda_{2k} = \lambda_{2k+1} \sim \frac{1}{2\pi (rm)^2} \int_0^{rm} xJ_k(x)^2 dx.
\]
But unlike Proposition~\ref{prop:diag}, we now have $rm = c$ of order 1, so that
\[
\lambda_0 \sim \frac{1}{2\pi c^2} \int_0^c xJ_0(x)^2 dx \gtrsim 1. 
\]
It follows that
\[
\frac{\tr(A^3) }{\tr(A^2)^{3/2} } \gtrsim 1.
\]
Thus the criterion of Proposition~\ref{prop:quads} fails, and we believe that instead of converging to a Gaussian, the standardized local integral $Z$ converges to a weighted sum of finitely many squares of Gaussians (with a number of degrees of freedom proportional to $c$, the higher values of $k$ being exponentially suppressed).

\section{Tail probability} \label{sec:tail}

One consequence of Theorem~\ref{thm:main} is that the random variable $X = \int_B \phi^2$ has tails satisfying
\[
\prob\left( |X - \E[X] | > y \sqrt{\var[X]} \right) \lesssim e^{-y^2/2}
\]
for any fixed $y >0$. Note that $\var[X] \asymp (rT)^{-n+1} \rightarrow 0$, and another approach is needed to control deviations where $|X - \E[X] |$ exceeds a given $\varepsilon > 0$ of larger size. One obtains such a bound using the foregoing estimates together with a Chernoff bound. Namely, for any $s \geq 0$ for which $\E[e^{sX}]$ is defined, the upper tail probability obeys
\begin{equation}
\prob(X - \E[X] > \varepsilon) \leq e^{-s\varepsilon - s\E[X]} \E\left[e^{sX} \right].
\end{equation}
The lower tail event $X - \E[X] < -\varepsilon$ can be treated with the same procedure applied to $-X$ instead of $X$.
The parameter $s$ is at our disposal, and the optimal choice is to minimize $\log \E\left[ e^{sX} \right] - s\E[X] - s\varepsilon$. As we saw in Section~\ref{sec:quad}, 
\[
\log \E\left[ e^{sX} \right] - s\E[X] - s\varepsilon = -s\varepsilon - s\sum_j \lambda_j - \frac{1}{2} \sum_j \log(1-2s\lambda_j)
\]
with both sides defined for $1-2s\lambda_{\max} > 0$, where $\lambda_{\max}$ is the largest $\lambda_j$.
Differentiating with respect to $s$, we find that the equation for the optimal $s$ is
\[
-\varepsilon - \sum_j \lambda_j + \sum_j \frac{\lambda_j}{1 - 2s\lambda_j}.
\]
Upon expanding $\lambda_j/(1-2s\lambda_j)$ in a geometric series, the first term $p=1$ cancels $\sum_j \lambda_j$ and we obtain
\[
\varepsilon = \sum_{p \geq 2} (2s)^{p-1} \sum_j \lambda_j^p.
\]
Truncating at $p=2$ suggests the choice
\begin{equation} \label{eqn:suggested-s}
s = \frac{\varepsilon}{2 \sum \lambda_j^2 }
\end{equation}
but it is not clear whether $1 - 2s\lambda_{\max} > 0$, that is, whether this choice of $s$ is valid.
Using our bound on the third moment
$\tr(A^3) \lesssim (rT)^{-3n/2+1}$
we conclude that
\begin{equation}
\lambda_{\max} \leq \left( \sum_j \lambda_j^3 \right)^{1/3} \lesssim (rT)^{-n/2+1/3}
\end{equation}

With the desired choice of $s$,
\[
2s\lambda_{\max} = \varepsilon \left( \sum_j \lambda_j^2 \right)^{-1} \lambda_{\max} \lesssim \varepsilon (rT)^{n/2 - 2/3} 
\]
by the third-moment bound for $\lambda_{\max}$ and the fact that $\sum \lambda_j^2$ is bounded above and below by constant multiples of $(rT)^{-n+1}$. It follows that the choice (\ref{eqn:suggested-s}) is valid provided
\[
\varepsilon (rT)^{n/2 - 2/3} \rightarrow 0
\]
or more generally when $\varepsilon$ is a sufficiently small constant multiple of $(rT)^{-n/2 + 2/3}$.

When $\varepsilon$ is small enough that (\ref{eqn:suggested-s}) is a valid choice of $s$, the Chernoff bound gives
\begin{align*}
\prob(X - E[X] \geq \varepsilon ) &\leq \exp\left( -s(\E[X]+\varepsilon) + \frac{1}{2} \sum_j -\log(1-2s\lambda_j) \right) \\
&\lesssim \exp \left(-\frac{1}{4} \varepsilon^2 \big( \sum \lambda_j^2 \big)^{-1} \right)
\end{align*}
assuming we can truncate the Taylor expansion at second order. This truncation is justified under the same condition allowing the choice of $s$ to begin with, since the higher terms can be bounded by
$
\varepsilon^3 \left( \sum_j \lambda_j^2 \right)^{-3} \sum_j \lambda_j^3. 
$
This completes the proof of Theorem~\ref{thm:tail}. \qed

Even when (\ref{eqn:suggested-s}) is not valid, one can obtain an upper bound by choosing a different $s$ (and (\ref{eqn:suggested-s}) is itself only an approximation to the optimal choice).
The workaround from \cite{dcit} is to instead take $s$ to be a small multiple of $\big(\sum \lambda_j^2 \big)^{-1/2}$, the validity of which follows from the bound $\lambda_{\max} \leq \sqrt{ \sum \lambda_j^2}$. 
This allows one to take, for example, $\varepsilon$ of order 1 at the cost of replacing $(rT)^{n-1}$ by $(rT)^{(n-1)/2}$ in Theorem~\ref{thm:tail}.

\section{Conclusion and outlook} \label{sec:conc}

As explained in Section~\ref{sec:variance}, the exponent $(rT)^{-n+1}$ in our lower bound on the variance is optimal. Less clear is the true size of the third moment. One approach to this would be to give an explicit treatment of higher-dimensional spheres along the lines of the semicircle law for $S^2$.

One of the aspects of the Central Limit Theorem is its universality: The normal distribution appears as a limit for a wide class of underlying distributions for the coefficients. In contrast, we have assumed Gaussian coefficients from the beginning, so it is perhaps less surprising that one recovers a Gaussian in the limit. An interesting problem would be to prove the CLT for random waves with other distributions for the coefficients $c_j$, not necessarily Gaussian. 
We expect that the central limit theorem will hold for a quite broad class of random coefficients.
Indeed, it might not be necessary to randomize at all if the underlying manifold $M$ has favourable dynamical properties.

Some further examples of coefficients seem to be within reach, namely if (a) the vector of coefficients is rotation-invariant, or (b) several moments of the coefficient distribution agree with those of a Gaussian. Assuming (a), the vector of coefficients is rotation-invariant and as above one can diagonalize the quadratic form $X = \sum_j \sum_k c_j c_k \int_B \phi_j \phi_k$ to obtain a sum of independent random variables. Lyapunov's criterion then applies and reduces the central limit theorem to estimates for the second and third moments as above. Assuming (b) up to moments of order six, these estimates follow from the direct calculation with Gaussians carried out here.
Note that the $p$-th moment of the quadratic form $X = \int_B \phi^2$ only involves moments up to order $2p$ of the coefficients of $\phi$.
For an example of rotation-invariant but non-Gaussian coefficients, one can take the vector of coefficients uniformly from the sphere $S^N$. 
However, the coefficients in this example are not independent.

We have given an upper bound on the probability that $\int_B \phi^2$ deviates from its average, at least in some ranges. An interesting problem would be to use the theory of large deviations to give a comparable lower bound on this probability. It could also be possible to prove the Central Limit Theorem directly from this machinery.

\section*{Acknowledgments}

We thank Yaiza Canzani for helpful discussions about Weyl's law.
MDCI thanks Peter Sarnak for his advice, encouragement, and support over the course of this work and thanks the Natural Sciences and Engineering Research Council of Canada for a PGS D grant. ML thanks Rupert L. Frank for multiple stimulating discussions on the topic, and thanks the Institute for Advanced Study for its hospitality during the 2017-2018 academic year.

\end{document}